\documentclass[12pt]{amsart}
\usepackage[latin9]{inputenc}
\usepackage{xcolor}
\usepackage{amstext}
\usepackage{amsthm}
\usepackage{amssymb}
\usepackage[unicode=true,
 bookmarks=false,
 breaklinks=false,pdfborder={0 0 1},backref=section,colorlinks=false]{hyperref}
\usepackage{xcolor}
\usepackage{verbatim}
\usepackage{amstext}
\usepackage{amsthm}
\usepackage{eurosym}
\usepackage{latexsym}
\usepackage{amsthm}
\usepackage{amsfonts}
\usepackage{ulem}

\makeatletter
\numberwithin{equation}{section}
\numberwithin{figure}{section}
\numberwithin{equation}{section}
\numberwithin{figure}{section}
\numberwithin{equation}{section}
\newtheorem{theorem}{Theorem}

\newtheorem{lemma}[theorem]{Lemma}
\newtheorem{corollary}[theorem]{Corollary}
\newtheorem{proposition}[theorem]{Proposition}

\theoremstyle{definition}
\newtheorem{definition}[theorem]{Definition}

\newtheorem*{acknowledgements*}{Acknowledgements}
\theoremstyle{remark}
\newtheorem{remark}[theorem]{Remark}

\numberwithin{theorem}{section}   
\thanks{The work of the first author was completed as a part of the implementation of the development program of the Scientific and Educational Mathematical Center Volga Federal District, agreement no. 075-02-2021-1393.}
\thanks{}
\thanks{}
\makeatother

\begin{document}
\title[Arazy-Cwikel property]{Arazy-Cwikel property for quasi-Banach couples}
\author{Sergey V. Astashkin}
\address{Astashkin:Department of Mathematics, Samara National Research
University, Moskovskoye shosse 34, 443086, Samara, Russia}
\email{astash56@mail.ru}
\author{Per G. Nilsson}
\address{Nilsson: Roslagsgatan 6, 113 55 Stockholm, Sweden}
\email{pgn@plntx.com}
\date{\today }

\begin{abstract}
The main result of this paper establishes that the known Arazy-Cwikel
property holds for classes of uniformly $K$-monotone spaces in the
quasi-Banach setting provided that the initial couple is mutually closed. As
a consequence, we get that the class of all quasi-Banach $K$-spaces (i.e.,
interpolation spaces which are described by the real $K$-method) with
respect to an arbitrary mutually closed Banach couple enjoys the
Arazy-Cwikel property. Another consequence complements some previous results
by Bykov and Ovchinnikov, showing that this property holds also for the
class of all interpolation quasi-Banach spaces with respect to a
quasi-Banach couple whenever all the couples involved have the uniform Calder\'{o}n-Mityagin property. We apply these results to some classical families of
spaces.
\end{abstract}

\maketitle

\section{Introduction}

\label{Intro}

According to the classical result of the interpolation theory of operators,
which was obtained independently by Calder\'{o}n \cite{CAL} and Mityagin \cite%
{Mit}, a Banach function space $X$ on an arbitrary underlying measure space
is an interpolation space with respect to the couple $\left( L^{1},L^{\infty
}\right) $ on that measure space if and only if the following monotonicity
property\footnote{%
Here we are using Calder\'{o}n's terminology. Mityagin formulates this result
somewhat differently.} holds: if $f\in X$, $g\in L^{1}+L^{\infty }$ and 
\begin{equation*}
\int\nolimits_{0}^{t}g^{\ast }\left( s\right) \,ds\leq
\int\nolimits_{0}^{t}f^{\ast }\left( s\right) \,ds,\;\;t>0
\end{equation*}%
(where $h^{\ast }$ denotes the nonincreasing left-continuous rearrangement
of $|h|$), then $g\in X$ and $\left\Vert g\right\Vert _{X}\leq \left\Vert
f\right\Vert _{X}$ (for all undefined terminology see the next section).
Since Peetre \cite{PeetreJ1963Nouv,PeetreJ1968brasilia} had proved (cf.~also
a similar result due independently to Oklander \cite{OklanderE1964}, and cf.
also \cite[ pp.~158--159]{KreeP1968}) that the functional $t\mapsto
\int\nolimits_{0}^{t}f^{\ast }\left( s\right) \,ds$ is in fact the $K$%
-functional of the function $f\in L^{1}+L^{\infty }$ for the couple $\left(
L^{1},L^{\infty }\right) $, the results of \cite{CAL,Mit} naturally led to
the introduction of the following definition. An intermediate Banach space $%
X $ with respect to a Banach couple $\overline{X}=\left( X_{0},X_{1}\right) $
is said to be \textit{{$K$}-monotone} with respect to $\overline{X}$ if
whenever elements $x\in X$ and $y\in \Sigma \left( \overline{X}\right) $
satisfy 
\begin{equation*}
{K}\left( t,y;\overline{X}\right) \leq {K}\left( t,x;\overline{X}\right) ,\;%
\text{for all\thinspace }\;t>0,
\end{equation*}%
it follows that $y\in X$. Moreover, if each interpolation Banach space with
respect to a Banach couple $\overline{X}$ is $K$-monotone, $\overline{X}$ is
said to have the Calder\'{o}n-Mityagin property. As is well known now this
property is shared by many Banach couples and, in particular, by each couple 
$\left( L^{p},L^{q}\right) $, $1\leq p<q\leq \infty $ (see e.g. \cite{SP78}
or \cite{Cwikel1}).

\vskip0.2cm

Using the Calder\'{o}n-Mityagin property of couples of $L^{p}$-spaces, Arazy and
Cwikel proved, in \cite{ArCw84}, that for all $1\leq p<q\leq \infty $ and
for each underlying measure space a Banach function space $X$ is an
interpolation space with respect to the couple $\left( L^{p},L^{q}\right) $
if and only if $X$ is such a space with respect to each of the couples $%
\left( L^{1},L^{q}\right) $ and $\left( L^{p},L^{\infty }\right) $, or more
formally 
\begin{equation}
Int\left( L^{p},L^{q}\right) =Int\left( L^{1},L^{q}\right) \cap Int\left(
L^{p},L^{\infty }\right)   \label{AC-statement-1}
\end{equation}%
($Int\left( X_{0},X_{1}\right) $ denotes the class of all interpolation
Banach spaces with respect to a Banach couple $\overline{X}=(X_{0},X_{1})$).
This result somewhat resembles the well-known interpolation Boyd theorem
together with its one-sided refinements. Indeed, the latter theorem reads
that, for $1\leq p<q\leq \infty $, the condition $1/q<\alpha \left( X\right)
\leq \beta \left( X\right) <1/p$ ($\alpha \left( X\right) $ and $\beta
\left( X\right) $ are the Boyd indices of $X$) ensures that a rearrangement
invariant space $X$ belongs to the set $Int\left( L^{p},L^{q}\right) $ \cite[%
Theorem~2.b.11]{LT79-II}, while in the extreme case when $1=p<q<\infty $
(resp. $1<p<q=\infty $) the same result holds for each $X$ such that $X\in
Int\left( L^{1},L^{\infty }\right) $ provided that only the one-sided
estimate $\alpha \left( X\right) >1/q$ (resp. $\beta \left( X\right) <1/p$)
is valid (see \cite{Mal81} and \cite{AM04}).

\vskip0.2cm

Later on, Bykov and Ovchinnikov obtained a result similar to %
\eqref{AC-statement-1} for families of interpolation spaces, corresponding
to weighted couples of shift-invariant ideal sequence spaces \cite{BO06}.
Moreover, let $X_{0}$ and $X_{1}$ be classical Lions-Peetre $K$-spaces with
respect to a Banach couple $(A_{0},A_{1})$, i.e., $X_{0}=(A_{0},A_{1})_{%
\alpha _{0},p_{0}}$, $X_{1}=(A_{0},A_{1})_{\alpha _{1},p_{1}}$, where $%
0<\alpha _{i}<1$, $1\leq p_{i}\leq \infty $, $i=0,1$. Then, from \cite[%
Theorem~2.9]{BO06}, in particular, it follows that for all $0<\theta <\eta
<1 $ and $0<p,q\leq \infty $ 
\begin{equation}
Int\left( \overline{X}_{\theta ,p},\overline{X}_{\eta ,q}\right) =Int\left(
X_{0},\overline{X}_{\eta ,q}\right) \cap Int\left( \overline{X}_{\theta
,p},X_{1}\right) .  \label{AC-statement-2}
\end{equation}

\vskip0.2cm

Recently some results of Arazy-Cwikel type were proved also in the
quasi-Banach setting. So, relation \eqref{AC-statement-1} has been extended
to the range $0\leq p<q\leq \infty $ and the classes of all interpolation
quasi-Banach spaces with respect to the couples of $L^{p}$-spaces of
measurable functions on the semi-axis $(0,\infty)$ with the Lebesgue measure \cite{CSZ} and the couples of sequence $l^{p}$-spaces \cite{AsCwNi21} (the definition of the extreme spaces $L^0$ and $l^0$ see also in \cite{CSZ} and \cite{AsCwNi21}).

\vskip0.2cm

On the other hand, already in \cite{ArCw84}, it was indicated that formula %
\eqref{AC-statement-2} fails to be valid for \textit{all} Banach couples $%
(X_{0},X_{1})$. In contrast to that, the main result of this paper
establishes that the Arazy-Cwikel property holds if we consider classes of 
\textit{uniformly $K$-monotone} (rather than of all interpolation) spaces
even in the quasi-Banach setting provided that the initial couple is
mutually closed. Since in the Banach case classes of $K$-monotone and $K$
spaces coincide \cite[Theorem~4.1.11]{BK91}, we immediately obtain that the class of all quasi-Banach $K$-spaces (i.e.,
interpolation spaces which are described by the real $K$-method) with
respect to an arbitrary mutually closed Banach couple enjoys the
Arazy-Cwikel property. Observe that the latter result has been announced, without a proof, in \cite{CwNi84} (see also \cite[p. 672]{BK91}). 

\vskip0.2cm

By the well-known theorem due to Cwikel \cite[Theorem~1]{Cwikel1}, for every
Banach couple $\overline{X}=(X_{0},X_{1})$ and all $0<\theta _{0},\theta
_{1}<1$, $1\leq p_{0},p_{1}\leq \infty $, the couple $\left( \overline{X}%
_{\theta _{0},p_{0}},\overline{X}_{\theta _{1},p_{1}}\right) $ is uniformly $%
K$-monotone. Hence, the above-mentioned result of \cite{BO06} suggests that
equality \eqref{AC-statement-2} holds for any uniformly $K$-monotone Banach
couple. Indeed, this statement is another consequence of our main theorem.
Moreover, we show that \eqref{AC-statement-2} is valid for an arbitrary
quasi-Banach couple $\overline{X}=(X_{0},X_{1})$ and all $0<\theta
_{0},\theta _{1}<1$, $0<p_{0},p_{1}\leq \infty $ whenever all the couples
involved in \eqref{AC-statement-2} have the uniform Calder\'{o}n-Mityagin
property.


\vskip0.2cm

Let us describe briefly the content of the paper. In Section \ref{Prel}, we
give preliminaries with basic definitions and notation. We address some
properties of quasi-Banach spaces and lattices, which relate to their
convexity, and necessary definitions and results from interpolation theory.
In the next section we collect some auxiliary results, many of which are
apparently to some extent known. Section \ref{Arazy-Cwikel} contains the main results of the paper (Theorem \ref{Ar-Cw1} and Corollaries
\protect\ref{Ar-Cw2-cor} and \protect\ref{Ar-Cw1-cor}) related to the
Arazy-Cwikel property for the Lions-Peetre $K$-spaces $X_{\theta ,p}$, $0<\theta <1$, $0<p\leq \infty $, in the quasi-Banach setting. In
conclusion, in Section \ref{ex1}, we apply these results to some classical families of spaces.

\vskip0.2cm

The authors would like to thank Professor Michael Cwikel, whose work and insight are behind the results announced in the paper \cite{CwNi84}.

\vskip0.3cm

\section{Preliminaries}

\label{Prel}

\subsection{Quasi-Banach spaces and lattices.}

\label{Prel4}

Recall that a (real) quasi-Banach space $X$ is a complete real vector space
whose topology is given by a quasi-norm $x\mapsto \|x\|$ which satisfy the
conditions: $\|x\|>0$ if $x\ne 0$, $\|\alpha x\|=|\alpha|\|x\|$, $\alpha\in%
\mathbb{R}$, $x\in X$, and $\|x_1+x_2\|_X\le C(\|x_1\|+\|x_2\|)$ for some $%
C>0$ and all $x_1,x_2\in X$.

\vskip0.2cm

If a quasi-Banach space $X$ is additionally a vector lattice such that $%
\Vert x\Vert \leq \Vert y\Vert $ whenever $|x|\leq |y|$, we say that $X$ is
a quasi-Banach lattice (see \cite{LT79-II}, \cite{MeyNie91}.)

\vskip0.2cm

Let $X$ be a quasi-Banach lattice, $x_{k}\in X$, $k=1,2,\dots ,n$. Then, any
element of the form $(\sum_{k=1}^{n}|x_{k}|^{p})^{1/p}$, $0<p<\infty $, can
be defined by means of a "homogeneous functional calculus" in the
quasi-Banach setting exactly as in the case of Banach lattices (cf. \cite[%
pp.~40-41]{LT79-II}, \cite{Kal84}, \cite{CT-86}). A quasi-Banach lattice $X$
is said to be \textit{(lattice) $p$-convex}, $0<p\leq \infty $, if for some
constant $M$ and any $x_{k}\in X$, $k=1,2,\dots ,n$ we have 
\begin{equation}
\label{equ2}
\Big\|\Big(\sum_{k=1}^{n}|x_{k}|^{p}\Big)^{1/p}\Big\|_{X}\leq M\Big(%
\sum_{k=1}^{n}\Vert x_{k}\Vert _{X}^{p}\Big)^{1/p}\;\;\mbox{if}\;\;p<\infty ,
\end{equation}%
with the usual modification for $p=\infty $ (see e.g. \cite{LT79-II} for $%
p\geq 1$ and \cite{Kal84} for $p>0$). By $M^{\left( p\right) }\left(
X\right) $ we will denote the minimal value of $M$ satisfying \eqref{equ2}.

\vskip0.2cm

Next, we will use the so-called $p$-convexification procedure which just an
abstract description of the mapping $f\mapsto |f|^{p}\mathrm{sign}\,f$ from $%
L_{r}(\mu )$, $0<r<\infty $, into $L_{rp}(\mu )$. Note that in a general
lattice $X$ there is no meaning to the symbol $x^{p}$ that makes us to
introduce new algebraic operations in $X$ (see \cite[pp.~53-54]{LT79-II}).

\vskip0.2cm

Let $X$ be a quasi-Banach lattice with the algebraic operations denoted by $%
+ $ and $\cdot $ and let $p>0$. For every $x,y\in X$ and $\alpha \in \mathbb{%
R} $ we define 
\begin{equation*}
x\oplus y:=(x^{1/p}+y^{1/p})^{p}\;\;\mbox{and}\;\;\alpha \odot x:=\alpha
^{p}\cdot x,
\end{equation*}%
where $\alpha ^{p}$ is $|\alpha |^{p}\mathrm{sign}\,\alpha $. Then, the set $%
X$, endowed with the operations $\oplus $, $\odot $ and the same order as in 
$X$ is a vector lattice, which denoted by $X^{(p)}$. Moreover, $%
|||x|||_{X^{(p)}}:=\Vert x\Vert _{X}^{1/p}$ is a lattice norm on $X^{(p)}$
and $(X^{(p)},|||\cdot|||_{X^{(p)}})$ is a quasi-Banach lattice for every $%
0<p<\infty$ \cite[Proposition~1.2]{CT-86}. One can easily check also that if 
$X$ is $r$-convex with the constant $M^{(r)}$, then $X^{(p)}$ is $pr$-convex
with the same constant.

\vskip0.2cm

Note that if $X$ is a quasi-Banach function lattice defined on some measure
space, $X^{(p)}$ can be identified with the space of functions $f$ such that 
$f^{p}:=|f|^{p}\mathrm{sign}\,f\in X$ equipped with the norm $|||f|||=\Vert
\,|f|^{p}\,\Vert ^{1/p}$.

\vskip0.2cm

\subsection{Interpolation of quasi-Banach spaces}

\label{Prel1}

Let us recall some basic constructions and definitions related to the
interpolation theory of operators. For more detailed information we refer to 
\cite{BSh,BL76,BK91,KPS82,Ovc84}.

\smallskip {} \vskip0.2cm

In this paper we are mainly concerned with interpolation within the class of
quasi-Banach spaces, while the linear bounded operators are considered as
the corresponding morphisms. A pair $\overline{X}=(X_{0},X_{1})$ of
quasi-Banach spaces is called a \textit{\ quasi-Banach couple} if $X_{0}$
and $X_{1}$ are both linearly and continuously embedded in some Hausdorff
topological vector space.

\vskip0.2cm

For each quasi-Banach couple $\overline{X}=(X_{0},X_{1})$ we define the 
\textit{\ intersection} $\,\Delta \left( \overline{X}\right) =X_{0}\cap
X_{1} $ and the \textit{sum} $\Sigma \left( \overline{X}\right) =X_{0}+X_{1}$
as the quasi-Banach spaces equipped with the quasi-norms 
\begin{equation*}
\Vert x\Vert _{\Delta \left( \overline{X}\right) }:=\max \left\{ \Vert
x\Vert _{X_{0}}\,,\,\Vert x\Vert _{X_{1}}\right\}
\end{equation*}%
and 
\begin{equation*}
\Vert x\Vert _{\Sigma \left( \overline{X}\right) }:=\inf \left\{ \Vert
x_{0}\Vert _{X_{0}}\,+\,\Vert x_{1}\Vert _{X_{1}}:\,x=x_{0}+x_{1},x_{i}\in
X_{i},i=0,1\right\} ,
\end{equation*}%
respectively.

\vskip0.2cm

A quasi-Banach space $X$ is called an \textit{intermediate space} for a
quasi-Banach couple $\overline{X}=(X_{0},X_{1})$ if the continuous
inclusions $\Delta \left( \overline{X}\right) \subset X\subset \Sigma \left( 
\overline{X}\right) $ hold. The set of intermediate spaces with respect to $%
\overline{X}$ will be denoted by $I\left(\overline{X}\right) $ or $I\left(
X_{0},X_{1}\right) $.

\vskip0.2cm

If $\overline{X}=(X_{0},X_{1})$ is a quasi-Banach couple, then we let $%
\mathfrak{L}\left( \overline{X}\right) $ (or $\mathfrak{L}(X_{0},X_{1})$)
denote the space of all linear operators $T:\,\Sigma \left( \overline{X}%
\right) \rightarrow \Sigma \left( \overline{X}\right) $ that are bounded on $%
X_{i}$, $i=0,1$, equipped with the quasi-norm 
\begin{equation*}
{\Vert T\Vert }_{\mathfrak{L}\left( \overline{X}\right)
}:=\max\limits_{i=0,1}{\ \Vert T\Vert }_{X_{i}\rightarrow X_{i}}.
\end{equation*}

\vskip0.2cm

Let $\overline{X}=(X_{0},X_{1})$ be a quasi-Banach couple and let $X\in
I\left( \overline{X}\right) $. Then, $X$ is said to be an \textit{\
interpolation space} with respect to the couple $\overline{X},$ if every
operator $T{\in \mathfrak{L}\left( \overline{X}\right) }$ is bounded on $X$.
Recall that, by the Aoki-Rolewicz theorem (see e.g. \cite[Lemma~3.10.1]{BL76}%
), every quasi-Banach space is a $F$-space (i.e., the topology in that space
is generated by a complete invariant metric). In particular, this applies to
the space ${\mathfrak{L}}({\overline{X}})$ which is obviously a quasi-Banach
space with respect to the quasi-norm $T\mapsto {\Vert T\Vert }_{\mathfrak{L}%
\left( \overline{X}\right)}$ and also with respect to the quasi-norm $%
T\mapsto \max({\Vert T\Vert }_{\mathfrak{L}\left( \overline{X}\right)},{%
\Vert T\Vert }_{X\to X})$ whenever the quasi-Banach space $X$ is an
interpolation space with respect to the quasi-Banach couple $\overline{X}%
=(X_0,X_1)$. As is well known (see e.g. \cite[Theorem~2.2.15]{Rud}), the
Closed Graph Theorem and the equivalent Bounded Inverse Theorem (see e.g. 
\cite[Corollary 2.2.12]{Rud}) hold for $F$-spaces. Therefore, exactly the same reasoning as required for the Banach case (see Theorem 2.4.2 of \cite[p.~28]{BL76}) shows that, if $X$ is an interpolation quasi-Banach space with respect to a quasi-Banach couple ${\overline{X}}=(X_{0},X_{1})$, then there exists a constant $C>0$ such that for every $T\in {\mathfrak{L}}({\overline{X}})$ we
have ${\Vert T\Vert }_{X\rightarrow X}\leq C{\Vert T\Vert }_{{\mathfrak{L}}({\overline{X}})}.$ The least constant $C$, satisfying the last inequality for
all such $T$, is called the \textit{interpolation constant} of $X$ with
respect to the couple $\overline{X}$. The collection of all interpolation
spaces with respect to the couple $\overline{X}$ will be denoted by $%
Int\left( \overline{X}\right) $ (or $Int\left( X_{0},X_{1}\right) )$.

\vskip0.2cm

One of the most important ways of constructing interpolation spaces is based
on use of the \textit{Peetre {$K$}-functional}, which is defined for an
arbitrary quasi-Banach couple $\overline{X}=(X_{0},X_{1})$, for every $x\in
\Sigma \left( \overline{X}\right) $ and each $t>0$ as follows: 
\begin{equation}
{K}(t,x;\overline{X}):=\inf
\{||x_{0}||_{X_{0}}+t||x_{1}||_{X_{1}}:\,x=x_{0}+x_{1},x_{i}\in {X_{i},i=0,1}%
\}.
\end{equation}%
For each fixed $x\in \Sigma \left( \overline{X}\right) $ one can easily show
that the function $t\mapsto ${$K$}$(t,x;\overline{X})$ is continuous,
non-decreasing, concave and non-negative on $\left( 0,\infty \right) $ \cite[%
Lemma~3.1.1]{BL76}.

\vskip0.2cm

Let $X$ be an intermediate quasi-Banach space with respect to a quasi-Banach
couple $\overline{X}=\left( X_{0},X_{1}\right) $. Then, $X$ is said to be a 
\textit{{$K$}-monotone space} with respect to the couple $\overline{X}$ if
whenever elements $x\in X$ and $y\in \Sigma \left( \overline{X}\right) $
satisfy 
\begin{equation*}
{K}\left( t,y;\overline{X}\right) \leq {K}\left( t,x;\overline{X}\right) ,\;%
\text{for all\thinspace }\;t>0,
\end{equation*}%
it follows that $y\in X$. If additionally $\left\Vert y\right\Vert _{X}\leq
C\left\Vert x\right\Vert _{X}$, for a constant $C$ which does not depend on $%
x$ and $y$, then we say that $X$ is a \textit{uniformly }$K-$\textit{monotone%
} space with respect to the couple $\overline{X}$. The infimum of all
constants $C$ with this property is referred as the \textit{{$K$}%
-monotonicity constant} of $X$. Clearly, each {$K$}-monotone space with
respect to the couple $\overline{X}$ is an interpolation space with respect
to this couple. As was already defined in Introduction, a couple $\overline{X%
}$\ has the Calder\'{o}n-Mityagin property whenever every interpolation
quasi-Banach space with respect to $\overline{X}$\ is $K$-monotone. The
collection of all uniformly $K-$monotone spaces with respect to the couple $%
\overline{X}$ will be denoted by $Int^{KM}\left( \overline{X}\right) $ (or $%
Int^{KM}\left( X_{0},X_{1}\right) $).

\vskip0.2cm

Suppose $E$ is a quasi-Banach function lattice on $(0,\infty )$ (with
respect to the usual Lebesgue measure and a.e. order) and $w:\,(0,\infty)\to%
\mathbb{R}$ is a nonnegative measurable function. Then, $E(w)$ is the
weighted quasi-Banach function lattice with the norm $\Vert x\Vert
_{E(w)}:=\Vert xw\Vert _{E}$. In particular, in what follows, by $\overline{%
L^{\infty }}$ we denote the couple $\left( L^{\infty},L^{\infty}(1/t)\right) 
$.

\vskip0.2cm

In particular, if $E=L^{p}\left( t^{-\theta },\frac{dt}{t}\right) $ (i.e., the weighted space $L^{p}\left( t^{-\theta }\right) $ on $(0,\infty )$ equipped with the measure $\frac{dt}{t}$), where $0<\theta <1,0<p\leq \infty$, we get the classical Lions-Peetre $K$-spaces $\overline{X}_{\theta ,p}$ endowed with the quasi-norms 
\begin{equation*}
\Vert x\Vert _{\overline{X}_{\theta ,p}}:=\left( \int_{0}^{\infty }\left(
K\left( t,x;\overline{X}\right) t^{-\theta }\right) ^{p}\,\frac{dt}{t}%
\right) ^{1/p}
\end{equation*}%
(with usual modification if $p=\infty $), see \cite{BL76}.

\vskip0.2cm

One of the most important properties of the real $K$-method is the following
reiteration theorem due to Brudnyi-Kruglyak (see e.g. \cite[Theorem 3.3.24]%
{BK91}). If $E_{0},E_{1}$ are quasi-Banach function lattices such that $%
E_{0},E_{1}\in Int\left( \overline{L^{\infty }}\right) $, then for every
quasi-Banach couple $\overline{X}=\left( X_{0},X_{1}\right) $ we have 
\begin{equation}  \label{reiteration1}
K\left( t,x;\overline{X}_{E_0:K},\overline{X}_{E_1:K}\right) \cong
K\left(t,K\left( \cdot ,x;X_{0},X_{1}\right) ;E_{0},E_1\right),\;\;t>0,
\end{equation}
with some constants independent of $x\in \overline{X}_{E_0:K}+\overline{X}%
_{E_1:K}$. This implies that for every quasi-Banach function lattice $F$ on $%
(0,\infty )$ 
\begin{equation}  \label{reiteration}
\left(\overline{X}_{E_0:K},\overline{X}_{E_1:K}\right)_{F:K}=\overline{X}%
_{E:K},
\end{equation}
where $E:=\left(E_0,E_1\right)_{F:K}$.

\vskip0.2cm

If $\overline{X}=(X_{0},X_{1})$ is a quasi-Banach couple and $X$ is an
intermediate space with respect to $\overline{X}$, then the \textit{relative
closure} $X^{c}$ of $X$ consists of all $x\in \Sigma \left( \overline{X}%
\right) $ for which there exists a bounded sequence $\left\{ x_{n}\right\}
\subset X,$ converging to $x$ in $\Sigma \left( \overline{X}\right) .$ The
norm in $X^{c}$ is taken as the infimum of all bounds of such sequences in $%
X $. Note that $K\left( t,x;\overline{X}\right) =K\left( t,x;\overline{X^{c}}%
\right) $ for all $x\in \Sigma \left( \overline{X}\right) $ and $t>0$, where 
$\overline{X^{c}}=\left( X_{0}^{c},X_{1}^{c}\right) $ (see \cite[Lemma 2]%
{Cwikel0}). We will say that a space $X\in I\left( \overline{X}\right) $ is 
\textit{relatively closed} in $\Sigma \left( \overline{X}\right) $ whenever $%
X=X^{c}$ with equivalence of norms\footnote{%
This differs from the usual definition saying that a space $X$ is relatively
closed in $\Sigma \left( \overline{X}\right) $ if $X^{c}=X$ isometrically
(see, for instance, \cite[Definition~I.1.4]{KPS82} or \cite[Definition~2.2.16%
]{BK91}).}. A quasi-Banach couple is said to be \textit{mutually closed} (or 
\textit{Gagliardo couple}) if both spaces $X_{0}$ and $X_{1}$ are relatively
closed in $\Sigma \left( \overline{X}\right) $.

\vskip0.2cm

Clearly, for every quasi-Banach couple $\overline{X}=(X_{0},X_{1})$ we have $%
X_{0}\subset \overline{X}_{L^{\infty }:K}$ and $X_{1}\subset \overline{X}%
_{L^{\infty }\left( 1/t\right) :K}$ with constant $1$. Moreover, a couple $%
\overline{X}=(X_{0},X_{1})$ is mutually closed if and only if the opposite
embeddings $\overline{X}_{L^{\infty }:K}\subset X_{0}$ and $\overline{X}%
_{L^{\infty }\left( 1/t\right) :K}\subset X_{1}$ hold, i.e., 
\begin{equation*}
\Vert x\Vert _{X_{0}}\cong \sup_{t>0}K(t,x;X_{0},X_{1})\;\;\mbox{and}%
\;\;\Vert x\Vert _{X_{1}}\cong \sup_{t>0}\frac{1}{t}K(t,x;X_{0},X_{1})
\end{equation*}%
with constants independent of $x\in (X_{0},X_{1})_{L^{\infty }:K}$ and $x\in
(X_{0},X_{1})_{L^{\infty }(1/t):K}$, respectively (see, for instance, \cite[%
Lemma 2.2.21, p.123]{BK91} or \cite[p.~384]{Ovc84}).

\vskip0.2cm

Let $\overline{X}=\left(X_{0},X_{1}\right)$ be a quasi-Banach couple and let 
$x\in X_{0}+X_{1}$, $x\neq0$. The \textit{orbit} ${\mathrm{Orb}}(x;\overline{%
X})$ of $x$ with respect to the class of operators $\mathfrak{L}(\overline{X}%
)$ is the linear space $\left\{ Tx:\,T\in\mathfrak{L}(\overline{X})\right\}$%
, which is equipped with the quasi-norm defined by 
\begin{equation*}
\left\Vert y\right\Vert _{\mathrm{Orb}\left(x\right)}:=\inf\left\{
\left\Vert T\right\Vert _{\mathfrak{L}(\overline{X})}:\;y=Tx,T\in \mathfrak{L%
}(\overline{X})\right\}.
\end{equation*}

Since any orbit ${\mathrm{Orb}}(x;\overline{X})$ can be regarded as a
quotient of the quasi-Banach space $\mathfrak{L}( \overline{X})$, it is a
quasi-Banach space itself. If for every nonzero $x\in X_{0}+X_{1}$ there
exists a linear functional $x^{\ast }\in \left( X_{0}+X_{1}\right) ^{\ast }$
with $\left\langle x,x^{\ast }\right\rangle \neq 0$ then $X_{0}\cap X_{1}$
is contained in $Orb\left( x;X_{0},X_{1}\right) $ continuously (see e.g. 
\cite[Section 1.6, p. 368]{Ovc84}). It is easy to see that then, moreover,
each orbit $Orb\left( x;\overline{X}\right) $ is an interpolation space
between $X_{0}$ and $X_{1}$.

\vskip0.2cm

A similar notion may be defined also by using the $K$-functional. If $%
\overline{X}=\left( X_{0},X_{1}\right) $ is a quasi-Banach couple, then the 
\textit{$K$}$\mathit{-orbit}$ ${{\mathrm{Orb}}}^{K}\left( x;\overline{X}%
\right) $ of an element $x\in \Sigma \left( \overline{X}\right) $, $x\neq 0$%
, is the space of all $y\in \Sigma \left( \overline{X}\right) $ such that
the quasi-norm 
\begin{equation*}
\left\Vert y\right\Vert _{{{\mathrm{Orb}}}^{K}(x)}:=\sup_{t>0}\frac{{K}%
\left( t,y;\overline{X}\right) }{{K}\left( t,x;\overline{X}\right) }
\end{equation*}%
is finite.

\bigskip \vskip0.2cm

It is obvious that for every quasi-Banach couple $\left( X_{0},X_{1}\right) $
and each $x\in X_{0}+X_{1}$ we have the continuous embedding ${\mathrm{Orb}}%
(x;\overline{X}){\subset }{{\mathrm{Orb}}}^{K}\left( x;\overline{X}\right) $
with constant $1$. A quasi-Banach couple $\overline{X}=\left(
X_{0},X_{1}\right) $ has the Calder\'{o}n-Mityagin property if and only if the
opposite embedding ${\mathrm{Orb}}^{K}(x;\overline{X}){\subset }{{\mathrm{Orb%
}}}\left( x;\overline{X}\right) $ holds for each $x\in X_{0}+X_{1}$, i.e.,
if for every $y\in {\mathrm{Orb}}^{K}(x;\overline{X})$ there exists an
operator $T\in \mathfrak{L}(\overline{X})$ such that $y=Tx$. Moreover, $%
\overline{X}=\left( X_{0},X_{1}\right) $ has the uniform Calder\'{o}n-Mityagin
property if and only if additionally we can choose $T\in \mathfrak{L}(%
\overline{X})$ so that $\Vert T\Vert _{\mathfrak{L}(\overline{X})}\leq
C\left\Vert y\right\Vert _{{\ {\mathrm{Orb}^{K}}}(x)}$, where $C$ is
independent of $x$ and $y$.

\vskip0.2cm

\subsection{Quasi-Banach lattice couples}

\label{Prel2}

Suppose $X_{0}$ and $X_{1}$ are two quasi-Banach lattices. We will say that $%
\overline{X}=(X_{0},X_{1})$ is a \textit{quasi-Banach lattice couple} if
there exists a Hausdorff topological vector lattice $\mathcal{H}$ such that
both $X_{0}$ and $X_{1}$ are embedded into $\mathcal{H}$ via a continuous,
interval preserving, lattice homeomorphism (see e.g. \cite{RaTr16}). Then,
one can easily check that the intersection $\Delta \left( \overline{X}%
\right) $ and the sum $\Sigma \left( \overline{X}\right) $ are quasi-Banach
lattices.

\vskip0.2cm

Suppose that a quasi-Banach lattice $X$ is intermediate with respect to a
quasi-Banach lattice couple $\overline{X}=(X_{0},X_{1})$ as a quasi-Banach
space. We will say that $X$ is an \textit{intermediate quasi-Banach lattice}
with respect to $\overline{X}$ if the canonical embeddings $I_{\Delta
}:\,\Delta \left( \overline{X}\right) \rightarrow X$ and $I_{\Sigma
}:\,X\rightarrow \Sigma \left( \overline{X}\right) $ are continuous,
interval preserving, lattice homeomorphisms (see \cite{BK91} and \cite%
{RaTr16}). \textbf{\ }By $I\left( \overline{X}\right) $ we will denote in
this case the set of all intermediate quasi-Banach lattices with respect to $%
\overline{X}$\textbf{.} 

\vskip0.2cm

Observe that every $K$-monotone quasi-Banach lattice with respect to a
quasi-Banach lattice couple is, in fact, uniform $K$-monotone. This can be
proved by a simple modification of the arguments for the Banach lattice case
used in \cite[Theorem 6.1]{CwNi03}.

\vskip0.2cm

We will say that a quasi-Banach lattice couple $\overline{X}=(X_{0},X_{1})$
is \textit{$p$-convex} if both spaces $X_{0}$ and $X_{1}$ are $p$-convex.

\vskip0.2cm

\subsection{The cone $Conv$ and $K$-functionals.}

\label{Prel3}

Let $Conv$ denote the cone of all continuous, concave, non-negative
functions defined on the half-line $(0,\infty )$. Each $f\in Conv$ is a
nondecreasing function such that the function $t\mapsto f\left( t\right) /t$
is nonincreasing. Consequently, $f\left( t\right) \leq \max \left(
1,t/s\right) f\left( s\right) $, $0<s,t<\infty$, and thus $Conv$ is a subset
of the space $\Sigma \left( \overline{L^{\infty }}\right)=L^{\infty }\left(
\min(1,1/s)\right)$ and for any $f\in Conv$ we have $\left\Vert f\right\Vert
_{\Sigma \left( \overline{L^{\infty }}\right) }=f\left( 1\right).$

\vskip0.2cm

Note that, for every quasi-Banach couple $\overline{X}$ and any $x\in \Sigma
\left( \overline{X}\right) $, the function $t\mapsto K\left( t,x;\overline{X}%
\right) $ belongs to the cone $Conv$. Moreover, for each function $h\in
\Sigma (\overline{L^{\infty }})$ the function $t\mapsto K\left( t,h;%
\overline{L^{\infty }}\right) $ is the least concave majorant of $|h|$ on $%
(0,\infty )$ \cite{Dm-74} (see also \cite[Proposition 3.1.17]{BK91}). Given
a set $\mathcal{U}\subseteq Conv$, a quasi-Banach couple $\overline{X}$ is
called $\mathcal{U}$-\textit{abundant} if for every $f\in \mathcal{U}$ there
exists $x\in \Sigma \left( \overline{X}\right) $ such that 
\begin{equation*}
K\left( t,x;\overline{X}\right) \cong f(t),\;\;t>0,
\end{equation*}%
with equivalence constants independent of $f$ (see \cite[Definition~4.4.8]%
{BK91})\footnote{%
This property is referred sometimes as the $K$-surjectivity of a couple, see
e.g. \cite[p. 217]{Nil83}.}. In particular, if $\mathcal{U}=\{K\left( \cdot
,x;\overline{Y}\right) ,x\in \Sigma \left( \overline{Y}\right) \}$ for some
quasi-Banach couple $\overline{Y}$, we will say that $\overline{X}$ is $%
\overline{Y}$-\textit{abundant}.

\vskip0.2cm

One can check easily that the couple $\overline{L^{\infty }}$ is $Conv$%
-abundant. Therefore, a quasi-Banach couple $\overline{X}$ is $\overline{%
L^{\infty }}$-abundant if and only if it is $Conv$-abundant.

\vskip0.2cm

The notation $A \preceq B$ means that there exists a positive constant $C$
with $A \leq C\cdot B$ for all applicable values of the arguments
(parameters) of the functions (expressions) $A$ and $B$. We will write $A
\cong B$ if $A \preceq B$ and $B \preceq A$.

\section{Auxiliary results}

\label{Aux}

\begin{lemma}
\label{renorming} Let $X$ be a $p$-convex quasi-Banach lattice, $p\in (0,1)$%
. Then the $1/p$-convexification $X^{\left( 1/p\right) }$ of $X$ has an
equivalent lattice norm and hence $X^{\left( 1/p\right) }$ is lattice
isomorphic to a Banach lattice.
\end{lemma}

\begin{proof}
Let $M^{(p)}$ be the $p$-convexity constant of $X$. Hence, if $x_{i}\in
X^{\left( 1/p\right) }$, $i=1,2\dots ,n$, then from definition of the $(1/p)$%
-convexification $X^{\left( 1/p\right) }$ (see Section \ref{Prel4}) it
follows 
\begin{equation*}
\left\Vert \sum_{i=1}^{n}\oplus x_{i}\right\Vert _{X^{\left( 1/p\right)
}}\leq (M^{(p)})^p\sum_{i=1}^{n}\left\Vert x_{i}\right\Vert
_{X^{\left(1/p\right) }}.
\end{equation*}

Now, a straightforward inspection shows that the functional 
\begin{equation*}
\left\Vert \left\Vert x\right\Vert \right\Vert :=\inf \left\{
\sum_{i=1}^{n}\left\Vert x_{i}\right\Vert _{X^{\left( 1/p\right)
}}:\left\vert x\right\vert \leq \sum_{i=1}^{n}\oplus \left\vert
x_{i}\right\vert ,x_{i}\in X^{\left( 1/p\right) }\right\}
\end{equation*}%
is a lattice norm on the space $X^{\left( 1/p\right) }$, which is equivalent
to the original quasi-norm.
\end{proof}

\vskip0.1cm

Some variants of the next result in the Banach case are well known (see e.g. 
\cite[Proposition $2.g.6$]{LT79-II}).

\begin{proposition}
\label{sum is p-convex} Suppose $\overline{X}=\left( X_{0},X_{1}\right) $ is
a $p$-convex quasi-Banach lattice couple, $p>0$. Then $\Sigma \left( 
\overline{X}\right) $ is a $p$-convex lattice and 
\begin{equation*}
M^{\left( p\right) }\left( \Sigma \left( \overline{X}\right) \right) \leq
\max (2^{1-1/p},2^{2/p-2})\max_{i=0,1}M^{\left( p\right) }\left(
X_{i}\right) .
\end{equation*}
\end{proposition}

\begin{proof}
We will assume that $0<p\leq 1$. The case when $p>1$ can be treated
similarly.

Let $x_{i}\in \Sigma \left( \overline{X}\right) $, $i=1,2,\dots ,n$, and let 
$x_{i}=x_{i}^{0}+x_{i}^{1}$ be an arbitrary representation, with $%
x_{i}^{0}\in X_{0},x_{i}^{1}\in X_{1}$. Then, we have 
\begin{eqnarray*}
\left( \sum_{i=1}^{n}\left\vert x_{i}\right\vert ^{p}\right) ^{1/p} &\leq
&\left( \sum_{i=1}^{n}\left( \left\vert x_{i}^{0}\right\vert ^{p}+\left\vert
x_{i}^{1}\right\vert ^{p}\right) \right) ^{1/p} \\
&\leq &2^{1/p-1}\left( \left( \sum_{i=1}^{n}\left\vert x_{i}^{0}\right\vert
^{p}\right) ^{1/p}+\left( \sum_{i=1}^{n}\left\vert x_{i}^{1}\right\vert ^{p}
\right) ^{1/p}\right) .
\end{eqnarray*}%
Therefore, 
\begin{eqnarray*}
\left\Vert \left( \sum_{i=1}^{n}\left\vert x_{i}\right\vert ^{p}\right)
^{1/p}\right\Vert _{\Sigma \left( \overline{X}\right) }^{p} &\leq
&2^{1-p}\left\Vert \left( \sum_{i=1}^{n}\left\vert x_{i}^{0}\right\vert
^{p}\right) ^{1/p}+\left( \sum_{i=1}^{n}\left\vert x_{i}^{1}\right\vert
^{p}\right) ^{1/p}\right\Vert _{\Sigma \left( \overline{X}\right) }^{p} \\
&\leq &2^{1-p}\left( \left\Vert \left( \sum_{i=1}^{n}\left\vert
x_{i}^{0}\right\vert ^{p}\right) ^{1/p}\right\Vert _{X_{0}}+\left\Vert
\left( \sum_{i=1}^{n}\left\vert x_{i}^{1}\right\vert ^{p}\right)
^{1/p}\right\Vert _{X_{1}}\right) ^{p} \\
&\leq &2^{1-p}\max_{i=0,1}M^{\left( p\right) }\left( X_{i}\right) ^{p}\left(
\left( \sum_{i=1}^{n}\left\Vert x_{i}^{0}\right\Vert _{X_{0}}^{p}\right)
^{1/p}+\left( \sum_{i=1}^{n}\left\Vert x_{i}^{1}\right\Vert
_{X_{1}}^{p}\right) ^{1/p}\right) ^{p} \\
&\leq &2^{1-p}\max_{i=0,1}M^{\left( p\right) }\left( X_{i}\right)
^{p}\sum_{i=1}^{n}\left( \left\Vert x_{i}^{0}\right\Vert
_{X_{0}}^{p}+\left\Vert x_{i}^{1}\right\Vert _{X_{1}}^{p}\right) \\
&\leq &2^{2-2p}\max_{i=0,1}M^{\left( p\right) }\left( X_{i}\right)
^{p}\left(\sum_{i=1}^{n}\left( \left\Vert
x_{i}^{0}\right\Vert_{X_{0}}+\left\Vert
x_{i}^{1}\right\Vert_{X_{1}}\right)\right)^{p}.
\end{eqnarray*}

Now passing in the right-hand side to the infimum over all admissible
representations $x_{i}=x_{i}^{0}+x_{i}^{1}$, $i=1,2,\dots ,n$, we obtain 
\begin{equation*}
\left\Vert \left( \sum_{i=1}^{n}\left\vert x_{i}\right\vert ^{p}\right)
^{1/p}\right\Vert _{\Sigma \left( \overline{X}\right) }^{p}\leq
2^{2-2p}\left( \max_{i=0,1}M^{\left( p\right) }\left( X_{i}\right) \right)
^{p}\sum_{i=1}^{n}\left\Vert x_{i}\right\Vert _{\Sigma \left( \overline{X}%
\right) }^{p},
\end{equation*}%
which completes the proof.
\end{proof}

The following simple fact is well-known in the case of Banach function
lattices (see e.g. \cite{Mal83}, \cite[Proposition 3.1.15]{BK91}, \cite%
{Mal89} and \cite[Lemma~1]{Mal13}).

\begin{lemma}
\label{Prop-K-functional}Let $\overline{X}$ be a quasi-Banach lattice
couple. Then for every $x\in \Sigma \left( \overline{X}\right) $ and all $%
t>0 $ we have 
\begin{eqnarray*}
K\left( t,x;\overline{X}\right) &=&K\left( t,|x|;\overline{X}\right) \\
&=&\inf \left\{ \left\Vert x_{0}\right\Vert _{X_{0}}+t\left\Vert
x_{1}\right\Vert _{X_{1}}:\left\vert x\right\vert =x_{0}+x_{1},0\leq
x_{0}\in X_{0},0\leq x_{1}\in X_{1}\right\} \\
&=&\inf \left\{ \left\Vert x_{0}\right\Vert _{X_{0}}+t\left\Vert
x_{1}\right\Vert _{X_{1}}:\left\vert x\right\vert \leq x_{0}+x_{1},0\leq
x_{0}\in X_{0},0\leq x_{1}\in X_{1}\right\} .
\end{eqnarray*}
\end{lemma}

\begin{proof}
Denote by $A(t,x)$ (resp. $B(t,x)$) the first (resp. second) infimum in the
statement of the lemma.

Let $x\in \Sigma \left( \overline{X}\right) $ and $|x|\leq x_{0}+x_{1}$, $%
0\leq x_{i}\in X_{i}$. Then, by the decomposition property (see e.g. \cite[%
p.~2]{LT79-II}), we may write $\left\vert x\right\vert =y_{0}+y_{1}$, where $%
0\leq y_{i}\leq \left\vert x_{i}\right\vert$, $i=0,1$. Hence, $y_{i}\in
X_{i} $ and $\left\Vert y_{i}\right\Vert _{X_{i}}\leq \left\Vert
x_{i}\right\Vert _{X_{i}},$ $i=0,1$. Thus, 
\begin{equation*}
\left\Vert x_{0}\right\Vert _{X_{0}}+t\left\Vert x_{1}\right\Vert
_{X_{1}}\geq \left\Vert y_{0}\right\Vert _{X_{0}}+t\left\Vert
y_{1}\right\Vert _{X_{1}}\geq A(t,x).
\end{equation*}%
Passing to the infimum in the left-hand side of this inequality, we deduce
that $B(t,x)\geq A(t,x)$. Since the opposite inequality is obvious, we get $%
A(t,x)=B(t,x)$.

Next, if $x\in \Sigma \left( \overline{X}\right) $ and $%
|x|=x_{0}+x_{1},x_{i}\in X_{i}$, then $\left\vert x\right\vert \leq
\left\vert x_{0}\right\vert +\left\vert x_{1}\right\vert $ and hence 
\begin{equation*}
\left\Vert x_{0}\right\Vert _{X_{0}}+t\left\Vert x_{1}\right\Vert
_{X_{1}}\geq B(t,x).
\end{equation*}%
Therefore, $K\left( t,|x|;\overline{X}\right) \geq B(t,x)$. It is clear also
that $K\left( t,|x|;\overline{X}\right) \leq A(t,x)$. Summing up, we have 
\begin{equation}
K\left( t,|x|;\overline{X}\right) =A(t,x)=B(t,x).  \label{Prop-K}
\end{equation}

Let now $x\in \Sigma \left( \overline{X}\right) $ and $x=x_{0}+x_{1}$, $%
x_{i}\in X_{i}$. Then $\left\vert x\right\vert \leq \left\vert
x_{0}\right\vert +\left\vert x_{1}\right\vert $, which implies as above that 
\begin{equation}
K\left( t,x;\overline{X}\right) \geq B(t,x).  \label{Prop-K1}
\end{equation}

Conversely, assume that $x\in \Sigma \left( \overline{X}\right) $, $%
|x|=x_{0}+x_{1}$ and $0\leq x_{i}\in X_{i}$. Denoting $x_{+}=x\vee 0$, $%
x_{-}=(-x)\vee 0$, we set $z_{0}:=x_{0}\wedge x_{+}-x_{0}\wedge x_{-}$ and $%
z_{1}:=x_{1}\wedge x_{+}-x_{1}\wedge x_{-}$. Clearly, $z_{i}\in X_{i}$ and,
since $x_{i}\leq |x|$ (see also \cite[Theorem~1.d.1]{LT79-II} or \cite[p.~142%
]{Kal84}), we have 
\begin{equation*}
\Vert z_{i}\Vert _{X_{i}}\leq \Vert x_{i}\wedge x_{+}+x_{i}\wedge x_{-}\Vert
_{X_{i}}=\Vert x_{i}\wedge |x|\Vert _{X_{i}}=\Vert x_{i}\Vert
_{X_{i}},\;\;i=0,1.
\end{equation*}%
Moreover, 
\begin{eqnarray*}
z_{0}+z_{1} &=&(x_{0}\wedge x_{+}+x_{1}\wedge x_{+})-(x_{0}\wedge
x_{-}+x_{1}\wedge x_{-}) \\
&=&(x_{0}+x_{1})\wedge x_{+}-(x_{0}+x_{1})\wedge x_{-}=x_{+}-x_{-}=x.
\end{eqnarray*}%
Thus, 
\begin{equation*}
\left\Vert x_{0}\right\Vert _{X_{0}}+t\left\Vert x_{1}\right\Vert
_{X_{1}}\geq \left\Vert z_{0}\right\Vert _{X_{0}}+t\left\Vert
z_{1}\right\Vert _{X_{1}}\geq K\left( t,x;\overline{X}\right) ,
\end{equation*}%
whence 
\begin{equation*}
A(t,x)\geq K\left( t,x;\overline{X}\right) .
\end{equation*}%
Combining this together with \eqref{Prop-K} and \eqref{Prop-K1}, we complete
the proof.
\end{proof}

\begin{proposition}
\label{K-functional for p-convexifications} Let $p>0$ and $\overline{X}%
=(X_0,X_1)$ be a $p$-convex quasi-Banach lattice couple. Then $\Sigma ( 
\overline{X^{\left( 1/p\right) }}) =\Sigma \left( \overline{X}\right)
^{\left(1/p\right) }$. Moreover, for arbitrary $x\in \Sigma ( \overline{%
X^{\left( 1/p\right) }})$ and all $t>0$ we have 
\begin{equation}
\min(2^{2p-2},2^{1-p})\min_{i=0,1}(M^{\left( p\right)
}\left(X_{i}\right))^{-p}K(t^{1/p},x;\overline{X})^{p}\le K(t,x; \overline{%
X^{\left( 1/p\right) }})  \label{K-funct}
\end{equation}
and 
\begin{equation}
K(t,x; \overline{X^{\left( 1/p\right) }})\le%
\max(2^{p-1},2^{1-p})K(t^{1/p},x; \overline{X})^{p}.  \label{K-funct1}
\end{equation}
\end{proposition}

\begin{proof}
For definiteness, we consider again only the case when $0<p\le 1$. Moreover,
by Lemma \ref{Prop-K-functional}, we can (and will) assume that $x\ge 0$.

Suppose first $x\in \Sigma ( \overline{X^{\left(1/p\right) }})$. Take any
decomposition $x=x_0\oplus x_1=(x_0^p+x_1^p)^{1/p}$ with $0\le x_{0}\in
X_0^{(1/p)},0\le x_{1}\in X_1^{(1/p)}$. Since $X_0$ and $X_1$ are $p$%
-convex, then, by Proposition \ref{sum is p-convex}, the sum $X_0+t^{1/p}X_1$
is $p$-convex for every $t>0$ with the constant $C_p:=\max(2^{1-1/p},
2^{2/p-2})\max_{i=0,1}M^{\left( p\right) }\left(X_{i}\right).$ Consequently, 
\begin{eqnarray*}
K(t^{1/p},x; \overline{X})^{p} &=&\|x\|_{X_0+t^{1/p}X_1}^p \\
&=&\|(x_0^p+x_1^p)^{1/p}\|_{X_0+t^{1/p}X_1}^p \\
&\le& C_p^p(\|x_0\|_{X_0+t^{1/p}X_1}^{p}+\|x_1\|_{X_0+t^{1/p}X_1}^{p}) \\
&\le& C_p^p(\|x_0\|_{X_0}^p+(t^{1/p}\|x_1\|_{X_1})^{p}) \\
&=& C_p^p(\|x_0\|_{X_0^{(1/p)}}+t\|x_1\|_{X_1^{(1/p)}}).
\end{eqnarray*}
Thus, $x\in \Sigma \left( \overline{X}\right) ^{\left(1/p\right) }$ and from
Lemma \ref{Prop-K-functional} it follows that 
\begin{equation*}
K(t^{1/p},x; \overline{X})^{p}\le C_p^p K(t,x; \overline{X^{\left(
1/p\right) }}),
\end{equation*}
which is equivalent to inequality \eqref{K-funct}.

Conversely, let $x\in \Sigma \left( \overline{X}\right) ^{\left( 1/p\right)
} $ and $x=y_{0}+y_{1}$, with $0\leq y_{0}\in X_{0}$, $0\leq y_{1}\in X_{1}$%
. Then, since $0<p\le 1$, we have $x\leq (y_{0}^{p}+y_{1}^{p})^{1/p}$.
Applying once more the decomposition property to the lattice $\Sigma \left( 
\overline{X}\right) ^{\left( 1/p\right) }$ (see \cite[p.~2]{LT79-II}), we
may write $x=(x_{0}^{p}+x_{1}^{p})^{1/p}$, where $0\leq x_{i}\leq y_{i}$, $%
i=0,1$. Since this implies that $x_{i}\in X_{i}^{(1/p)}$, $i=0,1$, it
follows that $x\in \Sigma (\overline{X^{\left( 1/p\right) }})$ and moreover 
\begin{eqnarray*}
K(t,x;\overline{X^{\left( 1/p\right) }}) &\leq &\Vert x_{0}\Vert
_{X_{0}^{(1/p)}}+t\Vert x_{1}\Vert _{X_{1}^{(1/p)}}=\Vert x_{0}\Vert
_{X_{0}}^{p}+t\Vert x_{1}\Vert _{X_{1}}^{p} \\
&\leq &\Vert y_{0}\Vert _{X_{0}}^{p}+t\Vert y_{1}\Vert _{X_{1}}^{p}\leq
2^{1-p}(\Vert y_{0}\Vert _{X_{0}}+t^{1/p}\Vert y_{1}\Vert _{X_{1}})^{p}.
\end{eqnarray*}%
Applying Lemma \ref{Prop-K-functional}, we see that this inequality implies %
\eqref{K-funct1}. Therefore, the proof is completed.
\end{proof}

\begin{proposition}
\label{parameter is a sum} If $\overline{X}=(X_{0},X_{1})$ is a quasi-Banach
couple, then for arbitrary $E_{2},E_{3}\in Int\left( \overline{L^{\infty }}%
\right) $ we have 
\begin{equation*}
\overline{X}_{E_{2}:K}+\overline{X}_{E_{3};K}=\overline{X}_{E_2+E_{3};K}.
\end{equation*}
\end{proposition}

\begin{proof}
Let $x\in \overline{X}_{E_{2}+E_{3};K}$. Then $K(\cdot ,x;\overline{X})\in
E_{2}+E_{3}$ and so, by Lemma \ref{Prop-K-functional}, $K(\cdot ,x;\overline{%
X})=f_{2}(\cdot )+f_{3}(\cdot )$, where $0\leq f_{2}\in E_{2},0\leq f_{3}\in
E_{3}$ and 
\begin{equation}
\Vert x\Vert _{\overline{X}_{E_{2}+E_{3};K}}\succeq \Vert f_{2}\Vert
_{E_{2}}+\Vert f_{3}\Vert _{E_{3}}  \label{when parameter is a sum}
\end{equation}%
with some constant independent of $x$. Moreover, since $E_{2},E_{3}\in
Int\left( \overline{L^{\infty }}\right) $, we may assume that $%
f_{2},f_{3}\in Conv$. Therefore, by a weak version of $K$-divisibility
property, which holds for quasi-Banach couples \cite[Theorem~3.2.12]{BK91},
we can find a decomposition $x=x_{2}+x_{3}$ such that $K\left( t,x_{i};%
\overline{X}\right) \leq \gamma f_{i}(t)$, $i=2,3$, for a universal constant 
$\gamma $ and all $t>0$. Combining these inequalities with 
\eqref{when
parameter is a sum}, we get 
\begin{eqnarray*}
\Vert x\Vert _{\overline{X}_{E_{2}+E_{3};K}} &\succeq &\Vert K\left( t,x_{2};%
\overline{X}\right) \Vert _{E_{2}}+\Vert K\left( t,x_{3};\overline{X}\right)
\Vert _{E_{3}} \\
&=&\Vert x_{2}\Vert _{\overline{X}_{E_{2};K}}+\Vert x_{3}\Vert _{\overline{X}%
_{E_{3};K}}\geq \Vert x\Vert _{\overline{X}_{E_{2}:K}+\overline{X}%
_{E_{3};K}},
\end{eqnarray*}%
whence 
\begin{equation*}
\overline{X}_{E_{2}+E_{3};K}\subset \overline{X}_{E_{2}:K}+\overline{X}%
_{E_{3};K}.
\end{equation*}

To prove the opposite embedding, assume that $x\in \overline{X}_{E_{2}:K}+%
\overline{X}_{E_{3};K}.$ Then, $x=x_2+x_3$, with $x_i\in \overline{X}%
_{E_{i}:K}$, $i=2,3$, and 
\begin{equation*}
\|x\|_{\overline{X}_{E_{2}:K}+\overline{X}_{E_{3};K}}\succeq \|x_2\|_{%
\overline{X}_{E_{2}:K}}+\|x_3\|_{\overline{X}_{E_{3}:K}}
\end{equation*}
with some constant independent of $x$. Therefore, since 
\begin{equation*}
K\left(t,x;\overline{X}\right)\le \max(C_0,C_1)\left(K\left( t,x_{2};%
\overline{X}\right)+K\left( t,x_{3};\overline{X}\right)\right),\;\;t>0,
\end{equation*}
where $C_i$ is the constant in the quasi-triangle inequality for $X_i$, $%
i=0,1$, we have 
\begin{eqnarray*}
\|x\|_{\overline{X}_{E_2+E_{3};K}} &=&\|K\left(\cdot,x;\overline{X}%
\right)\|_{E_2+E_{3}} \\
&\preceq& \|K\left(\cdot,x_2;\overline{X}\right)\|_{E_2}+\|K\left(\cdot,x_3;%
\overline{X}\right)\|_{E_3} \\
&=& \|x_2\|_{\overline{X}_{E_2;K}} +\|x_3\|_{\overline{X}_{E_3;K}}\preceq
\|x\|_{\overline{X}_{E_{2}:K}+\overline{X}_{E_{3};K}},
\end{eqnarray*}
which implies that 
\begin{equation*}
\overline{X}_{E_{2}:K}+\overline{X}_{E_{3};K}\subset \overline{X}%
_{E_2+E_{3};K},
\end{equation*}
and the proposition is proved.
\end{proof}

\vskip0.2cm

For some versions of the next result see \cite[Theorem 3.6.7, p. 415]{BK91}
and \cite[Theorem 3.20]{Nil82}. Recall that a quasi-Banach function lattice $%
E$ on a measure space $(T,\Sigma,\mu)$ \textit{has the Fatou property} if
from $x_n\in E,$ $n=1,2,\dots,$ $\sup_{n=1,2,\dots}\|x_n\|_E<\infty$ and $%
x_n\to{x}$ a.e. on $T $ it follows that $x\in E$ and $||x||_E\le
\liminf_{n\to\infty}{||x_n||_E}. $ Observe that such a lattice $E$ has the
Fatou property whenever from $x_n\in E,$ $x_n\ge 0$, $n=1,2,\dots,$ and $%
x_n\uparrow x$ a.e. on $T$ it follows that $x\in E$ and $\lim_{n\to\infty}%
\|x_n\|_E=\|x\|_E$.

\begin{proposition}
\label{Lemma-admissibility-1} Let $\overline{X}=(X_{0},X_{1})$ be a
quasi-Banach couple and let $X_{2}=\left( X_{0},X_{1}\right)
_{E_{2}:K},X_{3}=\left( X_{0},X_{1}\right)_{E_{3};K}$, where $E_{2},E_{3}\in
Int\left( \overline{L^{\infty }}\right) $. Suppose that at least one of the
following conditions holds:

(a) $\left( E_{2},E_{3}\right)$ is a mutually closed quasi-Banach couple;

(b) $E_{2}$ and $E_{3}$ are relatively complete spaces with respect to the
couple $\overline{L^{\infty }}$;

(c) $E_2$ and $E_3$ have the Fatou property.

Then, the couple $\left( X_{2},X_{3}\right)$ is mutually closed.
\end{proposition}

\begin{proof}
Assuming the condition (a) to be hold, we have 
\begin{equation*}
\Vert f\Vert _{E_{2}}\cong \sup_{t>0}K(t,f; E_{2},E_{3})\;\;\mbox{and}%
\;\;\Vert f\Vert _{E_{3}}\cong \sup_{t>0}\frac{1}{t}K(t,f; E_{2},E_{3}),
\end{equation*}%
with some constants independent of $f$ (see Section\ref{Prel1} or \cite[%
Lemma 2.2.21 p.123]{BK91} or \cite[p.~384]{Ovc84}). Moreover, by the
reiteration theorem (see equivalence \eqref{reiteration1} or \cite[%
Theorem~3.3.11]{BK91}), 
\begin{equation*}
K\left( t,x;X_{2},X_{3}\right) \cong K\left( t,K\left( \cdot ,x;\overline{X}%
\right) ;E_2,E_3\right) ,\;\;t>0, 
\end{equation*}%
with constants independent of $x\in X_{2}+X_{3} $. Therefore, 
\begin{eqnarray*}
\sup_{t>0}K\left( t,x;X_{2},X_{3}\right)&\cong& \sup_{t>0}K\left( t,K\left(
\cdot ,x;\overline{X}\right) ;E_2,E_3\right) \\
&\cong& \|K\left(\cdot,x;\overline{X}\right)\|_{E_{2}}=\|x\|_{E_{2}}
\end{eqnarray*}
and similarly 
\begin{eqnarray*}
\sup_{t>0}\frac1t K\left( t,x;X_{2},X_{3}\right)&\cong& \sup_{t>0}\frac1t
K\left( t,K\left( \cdot ,x;\overline{X}\right) ;E_2,E_3\right) \\
&\cong& \|K\left(\cdot,x;\overline{X}\right)\|_{E_{3}}=\|x\|_{E_{3}}.
\end{eqnarray*}
This implies that the couple $\left( X_{2},X_{3}\right)$ is mutually closed.

Since $E_{2},E_{3}\in Int\left( \overline{L^{\infty }}\right) $, then the
condition (a) is a consequence of (b). Hence, the result follows also in the
case (b).

Finally, assume that the condition (c) holds. Suppose that $\{x_n\}$ is a
sequence from $E_2$ such that $C:=\sup_{n=1,2,\dots}\|x_n\|_{E_2}<\infty$
and $x_n\to x$ in $\Sigma\left( \overline{L^{\infty }}\right)$. Then, $%
x_n\to x$ a.e. on $(0,\infty)$. Therefore, since $E_2$ has the Fatou
property, we obtain that $x\in E_2$ and $\|x\|_{E_2}\le C$. As a result, $%
E_{2}$ is a relatively complete space with respect to the couple $\overline{%
L^{\infty }}$. Since the same result holds for $E_3$, we get (b). This
completes the proof.
\end{proof}

In particular, since the space $L^{p}\left( t^{-\theta },\frac{dt}{t}\right) 
$ is relatively complete with respect to the couple $\overline{L^{\infty }}$
for all $0<\theta <1$ and $0<p\leq \infty $, we have

\begin{corollary}
\label{Cor-admissibility} For every quasi-Banach couple $\overline{X}%
=(X_{0},X_{1})$ and all $0<\theta_0,\theta_1 <1$, $0<p_0,p_1\leq \infty$ the
couple $\left( \overline{X}_{\theta_0 ,p_0},\overline{X}_{\theta_1,p_1}%
\right) $ is mutually closed.
\end{corollary}

\begin{remark}
The result of Corollary \ref{Cor-admissibility} is well known in the case
when $\overline{X}=(X_{0},X_{1})$ is a Banach couple and $1\le p_0,p_1\leq
\infty$. Indeed, then $\left( \overline{X}_{\theta_0 ,p_0},\overline{X}%
_{\theta_1,p_1}\right) $ is a $K$-monotone Banach couple \cite[Theorem~1]%
{Cwikel1} and hence it is mutually closed \cite[Lemma~3]{Cwikel0}.
\end{remark}

\vskip0.3cm

\section{Main results}

\label{Arazy-Cwikel}

In \cite{CwNi84}, it was announced a statement, asserting that the class of all $K$-spaces with respect to a Banach couple possesses the
Arazy-Cwikel property (see also \cite[p. 672]{BK91}). Below we give a proof of this result for mutually closed Banach couples (see Corollary 
\ref{Ar-Cw1-cor}). In fact, by using some ideas due to Bykov-Ovchinnikov \cite{BO06}, we obtain here some more general results of such type in the quasi-Banach setting.

\vskip0.2cm

The following main result of this paper establishes the Arazy-Cwikel
property for the class of $K$-monotone spaces with respect to an arbitrary
mutually closed quasi-Banach couple.

\begin{theorem}
\label{Ar-Cw1} Let $\overline{X}=(X_{0},X_{1})$ be a mutually closed
quasi-Banach couple. Then, for all $0<\theta <\eta <1$ and $0<p,q\leq \infty 
$ we have 
\begin{equation}  \label{main formula1}
Int^{KM}\left( \overline{X}_{\theta ,p},\overline{X}_{\eta ,q}\right)
=Int^{KM}\left( X_{0},\overline{X}_{\eta ,q}\right) \cap Int^{KM}\left( 
\overline{X}_{\theta ,p},X_{1}\right) .
\end{equation}
\end{theorem}

Since the classes of $K$-monotone and interpolation spaces for couples with the uniform Calder\'{o}n-Mityagin property coincide, we immediately get the following result, extending thereby some results due to Bykov-Ovchinnikov \cite{BO06} (see Introduction) to the quasi-Banach case.

\begin{corollary}
\label{Ar-Cw2-cor} Suppose $\overline{X}=(X_{0},X_{1})$ is a mutually closed
quasi-Banach couple such that the couples $\left( \overline{X}_{\theta ,p},%
\overline{X}_{\eta ,q}\right) $, $\left( X_{0},\overline{X}_{\eta ,q}\right) 
$ and $\left( \overline{X}_{\theta ,p},X_{1}\right) $ have the uniform Calder\'{o}n-Mityagin property for all $0<\theta <\eta <1$ and $0<p,q\leq \infty $.
Then, 
\begin{equation}
Int\left( \overline{X}_{\theta ,p},\overline{X}_{\eta ,q}\right) =Int\left(
X_{0},\overline{X}_{\eta ,q}\right) \cap Int\left( \overline{X}_{\theta
,p},X_{1}\right).  \label{main formula2}
\end{equation}
\end{corollary}

\begin{remark}
\label{rem-Ar-Cw2} In the Banach case the conditions of the last corollary
may be relaxed (see, for instance, \cite[Theorem 1]{Cwikel1},\cite[Theorem 2]%
{DmOv79},\cite[Theorem 4.17]{Nil83} and \cite[Theorem 4.4.18]{BK91}). In
particular, \eqref{main formula2} holds for every Banach couple $\overline{X}%
=(X_{0},X_{1})$ with the uniform Calder\'{o}n-Mityagin property for all $%
0<\theta <\eta <1$ and $1\leq p,q\leq \infty $.
\end{remark}

It seems to be unknown by now whether the classes of $K$-monotone and $K$-spaces coincide in the quasi-Banach setting. At the same time, it is well known that this coincidence holds in the Banach case \cite[Theorem~4.1.11]{BK91} (see also \cite[Corollary~4.3]{Nil83}). 
Therefore, from Theorem \ref{Ar-Cw1} it follows 

\begin{corollary}
\label{Ar-Cw1-cor} If $\overline{X}=(X_{0},X_{1}) $ is an arbitrary mutually
closed Banach couple, then 
\begin{equation*}
Int^{K}\left( \overline{X}_{\theta ,p},\overline{X}_{\eta ,q}\right)
=Int^{K}\left( X_{0},\overline{X}_{\eta ,q}\right) \cap Int^{K}\left( 
\overline{X}_{\theta ,p},X_{1}\right)
\end{equation*}
for all $0<\theta <\eta <1$ and $0<p,q\leq \infty $.
\end{corollary}

To prove the above results we need some additional notions and auxiliary
assertions.

\begin{lemma}
\label{add} Let $E$ and $F$ be quasi-Banach function lattices such that $%
E,F\in Int\left( \overline{L^{\infty }}\right) $. Then, the following
continuous embeddings hold: 
\begin{eqnarray*}
\left( L^{\infty }+F\right) \cap \left( E+L^{\infty }\left( 1/t\right)
\right) &\subset &E+F \\
\left( L^{\infty }\cap F\right) +\left( E\cap L^{\infty }\left( 1/t\right)
\right) &\supset &E\cap F.
\end{eqnarray*}
\end{lemma}

\begin{proof}
Suppose $x\in \left( L^{\infty }+F\right) \cap \left( E+L^{\infty }\left(
1/t\right) \right) $. Then, on the one hand, $x=x_{0}+x_{1}$, where $%
x_{0}\in L^{\infty }$ and $x_{1}\in F$. Since $F\in Int\left( \overline{%
L^{\infty }}\right) $, we have $x\chi _{\lbrack 1,\infty )}=x_{0}\chi
_{\lbrack 1,\infty )}+x_{1}\chi _{\lbrack 1,\infty )}\in F$ and $\Vert x\chi
_{\lbrack 1,\infty )}\Vert _{F}\preceq \Vert x\Vert _{L^{\infty }+F}$. On
the other hand, $x=x_{0}^{\prime }+x_{1}^{\prime }$, where $x_{0}^{\prime
}\in E$ and $x_{1}^{\prime }\in L^{\infty }(1/t)$. As above, we get $x\chi
_{\lbrack 0,1]}=x_{0}^{\prime }\chi _{\lbrack 0,1]}+x_{1}^{\prime }\chi
_{\lbrack 0,1]}\in E$ and $\Vert x\chi _{\lbrack 0,1]}\Vert _{E}\preceq
\Vert x\Vert _{E+L^{\infty }(1/t)}$. Consequently, $x=x\chi _{\lbrack
0,1]}+x\chi _{\lbrack 1,\infty )}\in E+F$ and $\Vert x\Vert _{E+F}\preceq
\Vert x\Vert _{\left( L^{\infty }+F\right) \cap \left( E+L^{\infty }\left(
1/t\right) \right) }$.

Since the second embedding can be proved quite similarly, we skip its proof.
\end{proof}

Let $\overline{X}=(X_{0},X_{1})$ be a mutually closed quasi-Banach couple
and let $X_{2}$, $X_{3}$ be quasi-Banach spaces such that $X_{2}\in
Int^{K}\left( X_{0},X_{3}\right) $, $X_{3}\in Int^{K}\left(
X_{2},X_{1}\right) $ and both couples $(X_{0},X_{3})$ and $(X_{2},X_{1}) $
are mutually closed. We will call such a collection $\{\overline{X}%
,X_{2},X_{3}\}$ \textit{admissible}.

\vskip0.2cm

\begin{remark}
\label{remark-admissibility} Further, the following sufficient condition for
the admissibility of a collection $\{\overline{X},X_{2},X_{3}\}$ will be
useful.

Assume that $\overline{X}=(X_{0},X_{1})$ is a mutually closed quasi-Banach
couple and $X_{2}$, $X_{3}$ are quasi-Banach spaces such that $X_{2}\in
Int^{K}\left( X_{0},X_{3}\right) $, $X_{3}\in Int^{K}\left(
X_{2},X_{1}\right) $. Then, by the reiteration theorem for the $K$-method
(see \eqref{reiteration} or \cite[Theorem~3.3.11]{BK91}), $X_{2},X_{3}\in
Int^{K}\left( X_{0},X_{1}\right) $ and hence there are parameters $%
E_{2},E_{3}\in Int\left( \overline{L^{\infty }}\right) $ with $X_{2}=\left(
X_{0},X_{1}\right) _{E_{2}:K},X_{3}=\left( X_{0},X_{1}\right) _{E_{3};K}$.
Moreover, by the hypothesis, $X_{0}=\left( X_{0},X_{1}\right) _{L^{\infty
}:K}$ and $X_{1}=\left( X_{0},X_{1}\right) _{L^{\infty}(1/t);K}$. Then,
applying Proposition \ref{Lemma-admissibility-1} to the couples $%
(X_{0},X_{3})$ and $(X_{2},X_{1})$, we conclude that they are mutually
closed. Hence, the collection $\{\overline{X},X_{2},X_{3}\}$ is admissible.
\end{remark}

\vskip0.2cm

Recall that $I\left( X,Y\right) $ denotes the set of all intermediate spaces
with respect to a quasi-Banach couple $(X,Y)$.

\vskip0.2cm

\begin{lemma}
\label{Lemma-transfer-admissbile} Let $\overline{X}=(X_0,X_1)$ be a
quasi-Banach couple and $\{\overline{X},X_{2},X_{3}\}$ be an admissible
collection. Then, we have 
\begin{equation*}
I\left( X_{2},X_{3}\right) =I\left( X_{0},X_{3}\right) \cap I\left(
X_{2},X_{1}\right).
\end{equation*}
\end{lemma}

\begin{proof}
Suppose first $X\in I\left( X_{2},X_{3}\right) $, i.e., $X_{2}\cap
X_{3}\subseteq X\subseteq X_{2}+X_{3}.$ Since $X_{2}\in Int\left(
X_{0},X_{3}\right) $, it follows that $X_{0}\cap X_{3}\subseteq X_2\subseteq
X_{0}+X_{3}$. Hence, $X_{0}\cap X_{3}\subseteq X\subseteq X_{0}+X_{3}$,
which means that $X\in I\left(X_{0},X_{3}\right) $. In the same way, $X\in
I\left(X_{2},X_{1}\right)$.

Conversely, assume that $X\in I\left( X_{0},X_{3}\right) \cap I\left(
X_{2},X_{1}\right) $, i.e., $X_{0}\cap X_{3}\subseteq X\subseteq X_{0}+X_{3}$
and $X_{2}\cap X_{1}\subseteq X\subseteq X_{2}+X_{1}.$ Then, in particular, 
\begin{equation*}
X\subseteq \left( X_{0}+X_{3}\right) \cap \left( X_{2}+X_{1}\right) .
\end{equation*}%
It is clear also that $X_{0}\subseteq \overline{X}_{L^{\infty }:K}$, $%
X_{1}\subseteq $ $\overline{X}_{L^{\infty }\left( 1/t\right) :K}$. Combining
this with the equalities $X_{2}=\overline{X}_{E_{2}:K}$, $X_{3}=\overline{X}%
_{E_{3};K}$, by Proposition \ref{parameter is a sum}, we conclude 
\begin{equation*}
\left( X_{0}+X_{3}\right) \cap \left( X_{2}+X_{1}\right) \subseteq \overline{%
X}_{L^{\infty }+E_{3}:K}\cap \overline{X}_{E_{2}+L^{\infty }\left(
1/t\right) :K}=\overline{X}_{E:K},
\end{equation*}%
where $E:=\left( L^{\infty }+E_{3}\right) \cap \left( E_{2}+L^{\infty
}\left( 1/t\right) \right) $. On the other hand, by Lemma \ref{add}, $%
E\subset E_{2}+E_{3}$, and so from Proposition \ref{parameter is a sum} it
follows 
\begin{equation*}
\overline{X}_{E:K}\subset \overline{X}_{E_{2}+E_{3}:K}=\overline{X}%
_{E_{2}:K}+\overline{X}_{E_{3}:K}=X_{2}+X_{3}.
\end{equation*}%
As a result, the last embeddings yield $X\subseteq X_{2}+X_{3}.$

Similarly, 
\begin{equation*}
X\supset X_{0}\cap X_{3}+X_{2}\cap X_{1}=\overline{X}_{E^{^{\prime }}:K},
\end{equation*}%
where $E^{^{\prime }}:=\left( L^{\infty }\cap E_{3}\right) +\left( E_{2}\cap
L^{\infty }\left( 1/t\right) \right) $. Again applying Lemma \ref{add}, we
get 
\begin{equation*}
\overline{X}_{E^{^{\prime }}:K}\supset \overline{X}_{E_{2}\cap E_{3}:K}=%
\overline{X}_{E_{2}:K}\cap \overline{X}_{E_{3}:K}=X_{2}\cap X_{3}
\end{equation*}%
and hence $X_{2}\cap X_{3}\subseteq X$. Thus, $X\in I\left(
X_{2},X_{3}\right) $, which completes the proof.
\end{proof}

\begin{lemma}
\label{L1} Let $\{\overline{X},X_{2},X_{3}\}$ be an admissible collection.
The following embeddings hold:

$\left( a\right) $ $Int^{K}\left( X_{2},X_{3}\right) \subseteq Int^{K}\left(
X_{0},X_{3}\right) \cap Int^{K}\left( X_{2},X_{1}\right) $;

$\left( b\right) $ $Int^{KM}\left( X_{2},X_{3}\right) \subseteq
Int^{KM}\left( X_{0},X_{3}\right) \cap Int^{KM}\left( X_{2},X_{1}\right) $.
\end{lemma}

\begin{proof}
Observe first that, thanks to Lemma \ref{Lemma-transfer-admissbile}, either
of the conditions $X\in Int^{K}\left( X_{2},X_{3}\right) $ and $X\in
Int^{KM}\left( X_{2},X_{3}\right) $ implies that 
\begin{equation*}
X\in I\left( X_{0},X_{3}\right) \cap I\left( X_{2},X_{1}\right) .
\end{equation*}

$\left( a\right) $ Suppose $X\in Int^{K}\left( X_{2},X_{3}\right) $, i.e., $%
X=\left( X_{2},X_{3}\right) _{E:K}$ for some $E\in Int\left( \overline{%
L^{\infty }}\right) $. According to the definition of an admissible
collection, $X_3=\left( X_{0},X_{3}\right) _{L^{\infty }(1/t):K}$ and there
are $F_{2},F_{3}\in Int\left( \overline{L^{\infty }} \right) $ such that $%
X_{2}=\left( X_{0},X_{3}\right) _{F_{2}:K},X_{3}=\left(X_{2},X_{1}\right)
_{F_{3}:K}$. Therefore, again by the reiteration theorem (see %
\eqref{reiteration1} or \cite[Theorem~3.3.11]{BK91}), with constants
independent of $x\in X_{2}+X_{3}$, we have 
\begin{equation}
K\left( t,x;X_{2},X_{3}\right) \cong K\left( t,K\left( \cdot
,x;X_{0},X_{3}\right) ;F_{2},L^{\infty }(1/t)\right) ,\;\;t>0.
\label{reiter1}
\end{equation}%
Consequently, $X=\left( X_{0},X_{3}\right) _{E_{X}:K}$, where $E_{X}=\left(
F_{2},L^{\infty }(1/t)\right) _{E:K}.$ Hence, $X\in Int^{K}\left(
X_{0},X_{3}\right) $. In the same way, $X\in Int^{K}\left(
X_{2},X_{1}\right) $.

$\left( b\right) $ Let $X\in Int^{KM}\left( X_{2},X_{3}\right) $. Suppose
that $x\in X$ and $y\in X_{0}+X_{3}$ satisfy 
\begin{equation*}
K\left( t,y;X_{0},X_{3}\right) \leq K\left( t,x;X_{0},X_{3}\right) ,\;\;t>0.
\end{equation*}%
Then, by \eqref{reiter1}, 
\begin{equation*}
K\left( 1,K\left( \cdot ,y;X_{0},X_{3}\right) ;F_{2},L^{\infty }(1/t)\right)
\leq CK\left( 1,K\left( \cdot ,x;X_{0},X_{3}\right) ;F_{2},L^{\infty
}(1/t)\right) <\infty ,
\end{equation*}%
that is, 
\begin{equation*}
K\left( \cdot ,y;X_{0},X_{3}\right) \in F_{2}+L^{\infty }(1/t).
\end{equation*}%
Thus, by Proposition \ref{parameter is a sum}, we obtain 
\begin{equation*}
y\in \left( X_{0},X_{3}\right) _{F_{2}+L^{\infty }(1/t):K}=\left(
X_{0},X_{3}\right) _{F_{2}:K}+\left( X_{0},X_{3}\right) _{L^{\infty
}(1/t):K}=X_{2}+{X}_{3}
\end{equation*}%
and hence from \eqref{reiter1} and similar equivalence for $y$ we infer 
\begin{equation*}
K\left( t,y;X_{2},X_{3}\right) \preceq K\left( t,x;X_{2},X_{3}\right)
,\;\;t>0.
\end{equation*}%
Therefore, by the assumption, $y\in X$ and so $X\in Int^{KM}\left(
X_{0},X_{3}\right) $. In the same way, $X\in Int^{KM}\left(
X_{2},X_{1}\right) $.
\end{proof}

The following definition is inspired by some results due to Bykov and
Ovchinnikov \cite{BO06}.

\begin{definition}
\label{additive K-orbits} We will say that an admissible collection $\{%
\overline{X},X_{2},X_{3}\}$, where $\overline{X}=(X_0,X_1)$, has \textit{%
additive $K$-orbits} whenever for each $x\in X_{2}+X_{3}$ the equality 
\begin{equation*}
Orb^{K}\left( x;X_{2},X_{3}\right) =Orb^{K}\left( x;X_{0},X_{3}\right)
+Orb^{K}\left( x;X_{2},X_{1}\right)
\end{equation*}%
holds uniformly with respect to $x$.
\end{definition}

\begin{lemma}
\label{L2} If an admissible collection $\{\overline{X},X_{2},X_{3}\}$ has
additive $K$-orbits, then 
\begin{equation*}
Int^{KM}\left( X_{0},X_{3}\right) \cap Int^{KM}\left( X_{2},X_{1}\right)
=Int^{KM}\left( X_{2},X_{3}\right) .
\end{equation*}
\end{lemma}

\begin{proof}
In view of Lemma \ref{L1}(b), it suffices only to prove the embedding 
\begin{equation*}
Int^{KM}\left( X_{0},X_{3}\right) \cap Int^{KM}\left( X_{2},X_{1}\right)
\subset Int^{KM}\left( X_{2},X_{3}\right) .
\end{equation*}

First, if $X\in Int^{KM}\left( X_{0},X_{3}\right) \cap Int^{KM}\left(
X_{2},X_{1}\right) $. Then, by Lemma \ref{Lemma-transfer-admissbile}, $X\in
I\left( X_{2},X_{3}\right) $. To prove that $X\in Int^{KM}\left(
X_{2},X_{3}\right) $, assume that $x\in X$ and $y\in X_{2}+X_{3}$ satisfy 
\begin{equation*}
K\left( t,y;X_{2},X_{3}\right) \leq K\left( t,x;X_{2},X_{3}\right) ,\;\;t>0.
\end{equation*}%
Therefore, $y\in Orb^{K}\left( x;X_{2},X_{3}\right) $ and $\Vert y\Vert
_{Orb^{K}\left( x;X_{2},X_{3}\right) }\leq 1$. Thus, by the assumption, $%
y=y_{3}+y_{2}$, where $y_{3}\in Orb^{K}\left( x;X_{0},X_{3}\right) $, $%
y_{2}\in Orb^{K}\left( x;X_{2},X_{1}\right) $, 
\begin{equation*}
K\left( t,y_{3};X_{0},X_{3}\right) \leq CK\left( t,x;X_{0},X_{3}\right)
,\;\;t>0,
\end{equation*}%
and 
\begin{equation*}
K\left( t,y_{2};X_{2},X_{1}\right) \leq CK\left( t,x;X_{2},X_{1}\right)
,\;\;t>0,
\end{equation*}%
with some constant $C$ independent of $x$. Since $X\in Int^{KM}\left(
X_{0},X_{3}\right) \cap Int^{KM}\left( X_{2},X_{1}\right) $, it follows that 
$y_{2},y_{3}\in X$. In consequence, $y\in X$ and $\Vert y\Vert _{X}\leq
C^{\prime }\Vert x\Vert _{X}$ with a constant $C^{\prime }$ independent of $x
$, which implies that $X\in Int^{KM}\left( X_{2},X_{3}\right) $.
\end{proof}

\begin{proposition}
\label{L4} Let $\{\overline{X},X_{2},X_{3}\}$ and $\{\overline{Y}%
,Y_{2},Y_{3}\}$ be two admissible collections such that $X_{2}=\overline{X}%
_{E_{2}:K}$, $Y_{2}=\overline{Y}_{E_{2}:K}$, $X_{3}=\overline{X}_{E_{3}:K}$
and $Y_{3}=\overline{Y}_{E_{3};K}$, where $E_{2},E_{3}\in Int\left( 
\overline{L^{\infty }}\right) $. If $\overline{Y}$ is an $\overline{X}$%
-abundant couple and the collection $\{\overline{Y},Y_{2},Y_{3}\}$ has
additive $K$-orbits, then the collection $\{\overline{X},X_{2},X_{3}\}$ has $%
K$-additive orbits as well.
\end{proposition}

\begin{proof}
Suppose first $x\in X_{2}+X_{3}$ and $y\in Orb^{K}\left(
x;X_{0},X_{3}\right) $. Then, arguing in the same way as in the proof of
Lemma \ref{L1} $\left( b\right) $ (see also \eqref{reiter1}), we conclude
that $y\in Orb^{K}\left( x;X_{2},X_{3}\right) $. Similarly, the fact that $%
y\in Orb^{K}\left( x;X_{2},X_{1}\right) $ yields $y\in Orb^{K}\left(
x;X_{2},X_{3}\right) $. Thus, 
\begin{equation*}
Orb^{K}\left( x;X_{0},X_{3}\right) +Orb^{K}\left( x;X_{2},X_{1}\right)
\subset Orb^{K}\left( x;X_{2},X_{3}\right) ,
\end{equation*}%
with constants independent of $x$. Therefore, it remains only to prove the
opposite embedding 
\begin{equation}
Orb^{K}\left( x;X_{2},X_{3}\right) \subset Orb^{K}\left(
x;X_{0},X_{3}\right) +Orb^{K}\left( x;X_{2},X_{1}\right) .  \label{embedd}
\end{equation}

Let $x\in X_{2}+X_{3}$ and $y\in Orb^{K}\left( x;X_{2},X_{3}\right) $ be
arbitrary. Since $x,y\in \Sigma \left( \overline{X}\right) $ and $\overline{Y%
}$ is $\overline{X}$-abundant, we can select elements $f,g\in \Sigma (%
\overline{Y})$ such that 
\begin{equation*}
K\left( t,x;\overline{X}\right) \cong K\left( t,f;\overline{Y}\right) \;\;%
\mbox{and}\;\;K\left( t,y;\overline{X}\right) \cong K\left( t,g;\overline{Y}%
\right) ,\;\;t>0,
\end{equation*}%
with universal constants. Moreover, by the conditions and reiteration
arguments, 
\begin{eqnarray*}
K\left( t,x;X_{2},X_{3}\right) &=& K\left( t,x;\overline{X}_{E_{2}:K},%
\overline{X}_{E_{3}:K}\right) \cong K\left( t,K\left( \cdot ,x;\overline{X}%
\right);E_{2},E_{3}\right) \\
&\cong & K\left( t,K\left( \cdot ,f;\overline{Y}\right);E_{2},E_{3}\right)=K%
\left( t,f;Y_{2},Y_{3}\right) ,\;\;t>0.
\end{eqnarray*}%
Similarly, we obtain 
\begin{equation*}
K\left( t,y;X_{2},X_{3}\right) \cong K\left( t,g;Y_{2},Y_{3}\right) ,\;\;t>0.
\end{equation*}%
Combining these equivalences with the assumption that $y\in Orb^{K}\left(
x;X_{2},X_{3}\right) $ we conclude that $g\in Orb^{K}\left(
f;Y_{2},Y_{3}\right) $. Therefore, since the collection $\{\overline{Y}%
,Y_{2},Y_{3}\}$ has additive $K$-orbits, we may write $g=g_{0}+g_{1}$, where 
$g_{0}\in Orb^{K}\left( f;Y_{0},Y_{3}\right)$, $g_{1}\in
Orb^{K}\left(f;Y_{2},Y_{1}\right) .$ Since $\overline{Y}=(Y_0,Y_1)$ is a
quasi-Banach couple, then the $K$-functional $K\left( t,g;\overline{Y}%
\right) $ is a quasi-norm on the sum $Y_0+Y_1$ with the constant $%
C=\max(C_0,C_1)$, where $C_i$ is the quasi-norm constant for $Y_i$, $i=0,1$.
Consequently, we have 
\begin{equation*}
K\left( t,y;\overline{X}\right) \cong K\left( t,g;\overline{Y}\right) \leq
C\cdot \left( K\left( t,g_{0};\overline{Y}\right) +K\left( t,g_{1};\overline{%
Y}\right) \right) ,\;\;t>0.
\end{equation*}
Using now once more a weak version of $K$-divisibility property \cite[%
Theorem~3.2.12]{BK91}, we can find a decomposition $y=y_{0}+y_{1}$ such that 
$K\left( t,y_{i};\overline{X}\right) \leq \gamma CK\left( t,g_{i};\overline{Y%
}\right) $, $i=0,1$, for all $t>0$. Combining these inequalities with the above relations and the hypothesis that $\{\overline{X},X_{2},X_{3}\}$ and $\{\overline{Y},Y_{2},Y_{3}\}$ are admissible collections, by  reiteration arguments, we get 
\begin{eqnarray*}
K\left( t,y_{0};X_{0},X_{3}\right) &=& K\left( t,K\left( \cdot ,y_0;%
\overline{X}\right);L_\infty,E_{3}\right)\preceq K\left( t,K\left( \cdot
,g_0;\overline{Y}\right);L_\infty,E_{3}\right) \\
&\cong & K\left(t,g_{0};Y_{0},Y_{3}\right)\preceq K\left(
t,f;Y_{0},Y_{3}\right) \cong K\left( t,K\left( \cdot ,f;\overline{Y}%
\right);L_\infty,E_{3}\right) \\
&\cong & K\left( t,K\left( \cdot ,x;\overline{X}\right);L_\infty,E_{3}%
\right)= K\left(t,x;X_{0},X_{3}\right),\;\;t>0.
\end{eqnarray*}

Thus, $y_{0}\in Orb^{K}\left( x;X_{0},X_{3}\right) .$ In the same way, $%
y_{1}\in Orb^{K}\left( x;X_{2},X_{1}\right) $. As a result, we get embedding %
\eqref{embedd} (with a constant independent of $x\in X_{2}+X_{3}$), and
therefore the proof is completed.
\end{proof}

\begin{lemma}
\label{L3} Let $\overline{X}=(X_{0},X_{1})$ be a $p$-convex quasi-Banach
lattice couple. Then, for each $r\geq 1/p$ and every $x\in \Sigma \left( 
\overline{X}\right) $ we have 
\begin{equation*}
Orb^{K}\left( x;\overline{X}\right) ^{\left( r\right) }=Orb^{K}\left( x;%
\overline{X^{\left( r\right) }}\right) .
\end{equation*}
\end{lemma}

\begin{proof}
Observe that, by Proposition \ref{K-functional for p-convexifications}, the
inequalities 
\begin{equation*}
K\left( t,y;\overline{X}\right) \leq CK\left( t,x;\overline{X}\right)
,\;\;t>0,
\end{equation*}%
and 
\begin{equation*}
K\left( t,y;\overline{X^{\left( r\right) }}\right) \leq C_{r}\cdot
K\left(t,x;\overline{X^{\left( r\right) }}\right) ,\;\;t>0,
\end{equation*}%
where the constant $C_r$ depends only on $\overline{X}$, $C$ and $r$, are
equivalent. As a result, the spaces $Orb^{K}\left( x;\overline{X}\right)
^{\left(r\right) }$ and $Orb^{K}( x;\overline{X^{\left( r\right) }}) $
coincide with equivalent norms.
\end{proof}

\begin{lemma}
\label{L3a} Let $0<\theta <1,0<p\le\infty $. Suppose $\overline{X}%
=(X_{0},X_{1})$ is a $p$-convex quasi-Banach lattice couple and $r$ is a
positive number such that $r\geq 1/p$. Then, $\left( \overline{X}_{\theta
,p}\right) ^{\left( r\right) }=\overline{X^{\left( r\right) }}_{\theta ,rp}$.
\end{lemma}

\begin{proof}
Assuming that $p<\infty$ and applying once more Proposition \ref%
{K-functional for p-convexifications}, we obtain 
\begin{eqnarray*}
\left\Vert x\right\Vert _{\left( \overline{X}_{\theta ,p}\right) ^{\left(
r\right) }}^{r} &=&\left( \int\nolimits_{0}^{\infty }\left( s^{-\theta
}K\left( s,x;\overline{X}\right) \right) ^{p}ds/s\right) ^{1/p} \\
&\cong &\left( \int\nolimits_{0}^{\infty }\left( s^{-\theta }K\left(
s^{1/r},x;\overline{X^{\left( r\right) }}\right) ^{r}\right)
^{p}\,ds/s\right) ^{1/p} \\
&\cong &\left( \left( \int\nolimits_{0}^{\infty }\left( u^{-\theta }K\left(
u,x;\overline{X^{\left( r\right) }}\right) \right) ^{rp}\,du/u\right)
^{1/rp}\right) ^{r} \\
&=&\Vert x\Vert _{\overline{X^{\left( r\right) }}_{\theta ,rp}}^{r}.
\end{eqnarray*}
The case when $p=\infty$ can be handled quite similarly.
\end{proof}

\begin{proof}[Proof of Theorem \protect\ref{Ar-Cw1}]
Let $\overline{X}=(X_{0},X_{1})$ be a mutually closed quasi-Banach couple, $%
0<\theta <\eta <1 $, $0<p,q\leq \infty $. We put $X_{2}=\overline{X}_{\theta
,p}$ and $X_{3}=\overline{X}_{\eta ,q}.$ Since the space $L_{\ast
}^{r}\left( t^{-s}\right)$ is relatively complete with respect to the couple 
$\overline{L^{\infty }}$ for all $0<s<1$ and $0<r\leq \infty$, by Remark \ref%
{remark-admissibility} (see also Proposition \ref{Lemma-admissibility-1}),
it follows that the collection $\{\overline{X},\overline{X}_{\theta ,p},%
\overline{X}_{\eta ,q}\}$ is admissible. Moreover, by the classical
reiteration theorem (see e.g. \cite[Theorem~3.5.3]{BL76}), we have 
\begin{equation}
\overline{X}_{\theta ,p}=\left( X_{0},\overline{X}_{\eta ,q}\right) _{\theta
/\eta ,p}\;\;\mbox{and}\;\;\overline{X}_{\eta ,q}=\left( \overline{X}%
_{\theta ,p},X_{1}\right) _{(\eta -\theta )/(1-\eta ),q}.
\label{general class formula}
\end{equation}

Along with the collection $\{\overline{X},\overline{X}_{\theta ,p},\overline{%
X}_{\eta ,q}\}$ we consider also the (admissible) collection $\{\overline{%
L^{\infty }},L_{\ast }^{p}\left( t^{-\theta }\right) ,L_{\ast }^{q}\left(
t^{-\eta }\right) \}$, where $L_{\ast }^{r}\left( t^{-s}\right) =L^{r}\left(
t^{-s},\frac{dt}{t}\right) $, $0<s<1$, $0<r\leq \infty $. Since for all $%
0<s<1$ and $0<r\leq \infty $ 
\begin{equation}
L_{\ast }^{r}\left( t^{-s}\right) =\overline{L^{\infty }}_{s,r},
\label{Class int}
\end{equation}%
then, applying \eqref{general class formula} to the couple $\overline{%
L^{\infty }}$, we get 
\begin{equation*}
L_{\ast }^{p}\left( t^{-\theta }\right) =\left( L^{\infty },L_{\ast
}^{q}\left( t^{-\eta }\right) \right) _{\theta /\eta ,p}
\end{equation*}%
and 
\begin{equation*}
L_{\ast }^{q}\left( t^{-\eta }\right) =\left( L_{\ast }^{p}\left( t^{-\theta
}\right) ,L^{\infty }(1/t)\right) _{(\eta -\theta )/(1-\eta ),q}.
\end{equation*}

In view of Lemma \ref{L2}, we need only to prove that the collection $\{%
\overline{X},X_{2},X_{3}\}$ has $K$-additive orbits. In turn, since the
couple $\overline{L^{\infty }}$ is $Conv$-abundant, according to Proposition %
\ref{L4}, this would be established, once we prove that the collection $\{%
\overline{L^{\infty }},L_{\ast }^{p}\left( t^{-\theta }\right) ,L_{\ast
}^{q}\left( t^{-\eta }\right) \}$ has $K$-additive orbits. In other words,
it suffices to show that 
\begin{eqnarray}
Orb^{K}\left( x;L_{\ast }^{p}\left( t^{-\theta }\right) ,L_{\ast }^{q}\left(
t^{-\eta }\right) \right) &=&Orb^{K}\left( x;L^{\infty },L_{\ast }^{q}\left(
t^{-\eta }\right) \right)  \notag \\
&+&Orb^{K}\left( x;L_{\ast }^{p}\left( t^{-\theta }\right) ,L^{\infty
}(1/t)\right),  \label{additive}
\end{eqnarray}%
for every $x\in L_{\ast }^{p}\left( t^{-\theta }\right) +L_{\ast
}^{q}\left(t^{-\eta }\right) $ and $0<\theta<\eta<1$, $0<p,q\le \infty$.

Let $r\geq \max (1/p,1/q,1)$. Then, from \eqref{Class int} and Lemmas \ref%
{L3}, \ref{L3a} it follows 
\begin{eqnarray*}
Orb^{K}\left( x;L_{\ast }^{p}\left( t^{-\theta }\right) ,L_{\ast }^{q}\left(
t^{-\eta }\right) \right) ^{(r)} &=&Orb^{K}\left( x;\left( \overline{%
L^{\infty }}_{\theta ,p}\right) ^{(r)},\left( \overline{L^{\infty }}_{\eta
,q}\right) ^{(r)}\right)  \\
&=&Orb^{K}\left( x;\left( \overline{L^{\infty }}^{(r)}\right) _{\theta
,pr},\left( \overline{L^{\infty }}^{(r)}\right) _{\eta ,qr}\right) ,
\end{eqnarray*}%
and similarly 
\begin{equation*}
Orb^{K}\left( x;L^{\infty },L_{\ast }^{q}\left( t^{-\eta }\right) \right)
^{(r)}=Orb^{K}\left( x;\left( L^{\infty }\right) ^{(r)},\left( \overline{%
L^{\infty }}^{(r)}\right) _{\eta ,qr}\right) ,
\end{equation*}%
\begin{equation*}
Orb^{K}\left( x;L_{\ast }^{p}\left( t^{-\theta }\right) ,L^{\infty
}(1/t)\right) ^{(r)}=Orb^{K}\left( x;\left( \overline{L^{\infty }}%
^{(r)}\right) _{\theta ,pr},\left( L^{\infty }(1/t)\right) ^{(r)}\right) .
\end{equation*}%
Observe that $\left( L^{\infty }\right) ^{(r)}=L^{\infty }$ and $\left(
L^{\infty }(1/t)\right) ^{(r)}=L^{\infty }(t^{-1/r})$. Therefore, since $%
pr\geq 1$, $qr\geq 1$, applying \cite[Theorem~2.5]{BO06} to the Banach
couple $\left( L^{\infty },L^{\infty }(t^{-1/r})\right) $, we obtain 
\begin{eqnarray*}
Orb\left( x;\left( \overline{L^{\infty }}^{(r)}\right) _{\theta ,pr},\left( 
\overline{L^{\infty }}^{(r)}\right) _{\eta ,qr}\right)  &=&Orb\left(
x;\left( L^{\infty }\right) ^{(r)},\left( \overline{L^{\infty }}%
^{(r)}\right) _{\eta ,qr}\right)  \\
&+&Orb\left( x;\left( \overline{L^{\infty }}^{(r)}\right) _{\theta
,pr},\left( L^{\infty }(1/t)\right) ^{(r)}\right) .
\end{eqnarray*}%
Note that all the couples involved here have the uniform Calder\'{o}n-Mityagin
property (see, for instance, \cite{Cwikel1} or \cite{SP78}). Hence (see also
Section \ref{Prel1}), we can replace in the last equality the orbits with
the $K$-orbits: 
\begin{eqnarray*}
Orb^{K}\left( x;\left( \overline{L^{\infty }}^{(r)}\right) _{\theta
,pr},\left( \overline{L^{\infty }}^{(r)}\right) _{\eta ,qr}\right) 
&=&Orb^{K}\left( x;\left( L^{\infty }\right) ^{(r)},\left( \overline{%
L^{\infty }}^{(r)}\right) _{\eta ,qr}\right)  \\
&+&Orb^{K}\left( x;\left( \overline{L^{\infty }}^{(r)}\right) _{\theta
,pr},\left( L^{\infty }(1/t)\right) ^{(r)}\right) ,
\end{eqnarray*}%
and combining this with the preceding formulae, we get 
\begin{eqnarray}
Orb^{K}\left( x;L_{\ast }^{p}\left( t^{-\theta }\right) ,L_{\ast }^{q}\left(
t^{-\eta }\right) \right) ^{(r)} &=&Orb^{K}\left( x;L^{\infty },L_{\ast
}^{q}\left( t^{-\eta }\right) \right) ^{(r)}  \notag \\
&+&Orb^{K}\left( x;L_{\ast }^{p}\left( t^{-\theta }\right) ,L^{\infty
}(1/t)\right) ^{(r)}.  \label{additive1}
\end{eqnarray}

Next, since $1/r\le \min(p,q)$, the quasi-Banach function couples $%
\left(L^{\infty },L_{\ast }^{q}\left(t^{-\eta }\right) \right)$ and $%
\left(L^{\infty },L_{\ast }^{q}\left(t^{-\eta }\right) \right)$ are $1/r$%
-convex. Hence, 
\begin{equation*}
U:=Orb^{K}\left( x;L^{\infty },L_{\ast }^{q}\left(t^{-\eta }\right)
\right)\;\;\mbox{and}\;\;V:=Orb^{K}\left( x;L_{\ast }^{p}\left( t^{-\theta
}\right) ,L^{\infty}(1/t)\right)
\end{equation*}
are $1/r$-convex quasi-Banach function lattices. Indeed, since the spaces $%
L^{\infty }$ and $L_{\ast }^{q}\left(t^{-\eta }\right)$ are $1/r$-convex
with constant $1$, in view of Lemma \ref{sum is p-convex}, for any $%
y_1,y_2,\dots,y_n\in U$ we have 
\begin{eqnarray*}
\Big\|\Big(\sum_{k=1}^n |y_k|^{1/r}\Big)^{r}\Big\|_{U} &=& \sup_{t>0}\frac{%
K\left(t,\Big(\sum_{k=1}^n |y_k|^{1/r}\Big)^{r};L^{\infty },L_{\ast
}^{q}\left(t^{-\eta }\right) \right)}{K\left(t,x;L^{\infty },L_{\ast
}^{q}\left(t^{-\eta }\right) \right)} \\
&\le& 2^{r-1}\sup_{t>0}\left(\frac{\sum_{k=1}^n K\left(t,|y_k|;L^{\infty
},L_{\ast }^{q}\left(t^{-\eta }\right)\right)^{1/r}}{K\left(t,x;L^{\infty
},L_{\ast }^{q}\left(t^{-\eta }\right)\right)^{1/r}}\right)^r \\
&\le& 2^{r-1}\left(\sum_{k=1}^n\left(\sup_{t>0}\frac{K\left(t,|y_k|;L^{%
\infty },L_{\ast }^{q}\left(t^{-\eta }\right)\right)}{K\left(t,x;L^{\infty
},L_{\ast }^{q}\left(t^{-\eta }\right)\right)} \right)^{1/r}\right)^r \\
&=&2^{r-1}\left(\sum_{k=1}^n\|y_k\|_U^{1/r}\right)^r.
\end{eqnarray*}%
Quite similarly, for any $z_1,z_2,\dots,z_n\in V$, 
\begin{equation*}
\Big\|\Big(\sum_{k=1}^n |z_k|^{1/r}\Big)^{r}\Big\|_{V}\le
2^{r-1}\left(\sum_{k=1}^n\|z_k\|_V^{1/r}\right)^r.
\end{equation*}
From this observation and Proposition \ref{K-functional for
p-convexifications} it follows that $(U+V)^{\left( r\right) } =U^{\left(
r\right) }+V^{\left( r\right) }$. Combining this together with %
\eqref{additive1} we get \eqref{additive}, which completes the proof of the
theorem.
\end{proof}

\begin{remark}
In the paper \cite{ArCw84}, by using some ideas from \cite{Ovc82}, it is
constructed an example of (Banach) rearrangement invariant spaces of
measurable functions $B_{1}$, $B_{p}$, $B_{q}$, $B_{\infty }$ and $A$ such
that $B_{p}=\left( B_{1},B_{\infty }\right) _{1/4,p}$, $B_{q}=\left(
B_{1},B_{\infty }\right) _{3/4,q}$, $A$ is an interpolation space with
respect to both $\left( B_{1},B_{q}\right) $ and $\left( B_{q},B_{\infty
}\right) $, but $A$ fails to be an interpolation space with respect to $%
\left( B_{p},B_{q}\right) $. This indicates that the Arazy-Cwikel formula %
\eqref{main formula2} cannot be extended to the class of all Banach couples
and hence the conditions of Corollary \ref{Ar-Cw2-cor}, in general, cannot
be skipped even in the Banach case. At the same time, note that the example
of $l^{p}$-spaces, $0<p\leq \infty $, shows that the latter conditions are
not necessary to have \eqref{main
formula2} (see Example \ref{ex1} below or \cite{AsCwNi21}). 
\end{remark}

\begin{remark}
By the well-known Wolff's theorem \cite{Wolff} (see also \cite{JNP} and \cite%
[Theorem~4.5.16]{BK91}), if $X_{2},X_{3}$ are intermediate spaces with
respect to a quasi-Banach couple $\overline{X}=\left( X_{0},X_{1}\right) $
such that $\left( X_{0},X_{3}\right) _{\mu ,p}=X_{2},\left(
X_{2},X_{1}\right) _{\nu,q}=X_{3}$ for some $0<\mu ,\nu <1$ and $0<p,q\leq
\infty $, then $X_{2}=\overline{X}_{\theta ,p}$ and $X_{3}=\overline{X}%
_{\eta ,q}$ for some $0<\theta <\eta <1.$ This allows to relax somewhat the
assumptions of Theorem \ref{Ar-Cw1} and Corollary \ref{Ar-Cw1-cor}. For
instance, we obtain that the equality 
\begin{equation*}
Int^{KM}\left( X_{2},X_{3}\right) =Int^{KM}\left( X_{0},X_{3}\right) \cap
Int^{KM}\left( X_{2},X_{1}\right)
\end{equation*}%
holds for intermediate spaces $X_{2},X_{3}$ with respect to a quasi-Banach
couple $\overline{X}=\left( X_{0},X_{1}\right) $ whenever $\left(
X_{0},X_{3}\right) _{\mu ,p:K}=X_{2}$ and $\left( X_{2},X_{1}\right) _{\nu
,q:K}=X_{3}$ for some $0<\mu ,\nu <1$, $0<p,q\leq \infty $.
\end{remark}

\vskip0.4cm

\section{Some applications}

\label{ex1}

\subsection{$l^{p}$-spaces.}

As usual, the space $l^{p}$, $0<p\leq \infty $, consists of all sequences $%
x=(x_{k})_{k=1}^{\infty }$ such that 
\begin{equation*}
\Vert x\Vert _{l^{p}}:=\Big(\sum_{k=1}^{\infty }|x_{k}|^{p}\Big)%
^{1/p}<\infty 
\end{equation*}%
(with the usual modification for $p=\infty $).

Let $0<s<p<q<r\leq \infty .$ Since the couple $\left( l^{s},l^{r}\right) $
is mutually closed, from Theorem \ref{Ar-Cw1} it follows that 
\begin{equation}
Int^{KM}\left( l^{p},l^{q}\right) =Int^{KM}\left( l^{s},l^{q}\right) \cap
Int^{KM}\left( l^{p},l^{r}\right) .  \label{lp-case}
\end{equation}%
Note that, if $q\geq 1$, then all the couples from \eqref{lp-case} have the
uniform Calder\'{o}n-Mityagin property \cite[Corollary~4.6]{AsCwNi21} and hence
the corresponding classes of $K$-monotone and interpolation spaces coincide.
Therefore, in this case we have 
\begin{equation}
Int\left( l^{p},l^{q}\right) =Int\left( l^{s},l^{q}\right) \cap Int\left(
l^{p},l^{r}\right) .  \label{lp-case1}
\end{equation}%
Moreover, by \cite[Theorem~4.1]{AsCwNi21}, \eqref{lp-case1} holds for all $%
0<s<p<q<r\leq \infty .$ At the same time, if $q<1$, $\left(
l^{p},l^{q}\right) $ is not a Calder\'{o}n-Mityagin couple \cite[Theorem~5.3]%
{AsCwNi21}, and so $Int\left( l^{p},l^{q}\right) \neq Int^{KM}\left(
l^{p},l^{q}\right) $ and $Int\left( l^{s},l^{q}\right) \neq Int^{KM}\left(
l^{s},l^{q}\right) $. Thus, if $q<1$, equalities \eqref{lp-case} and %
\eqref{lp-case1} are different. 

Next, let us consider a more general case.

\vskip0.2cm

\subsection{Lorentz $L^{p,q}$-spaces.}

Let $L^{p,q}$, $0<p<\infty $, $0<q\leq \infty $, be the Lorentz spaces of
measurable functions on an arbitrary underlying $\sigma $-finite measure
space $(T,\Sigma ,\mu )$, which can be defined using the quasi-norms 
\begin{equation*}
\Vert f\Vert _{L^{p,q}}:=\Big(\int_{0}^{\infty }f^{\ast }(u)^{q}\,d(u^{q/p})%
\Big)^{1/q}\;\;\mbox{for}\;\;q<\infty ,
\end{equation*}%
and 
\begin{equation*}
\Vert f\Vert _{L^{p,\infty }}:=\mathrm{ess\,sup}_{u>0}\left( f^{\ast
}(u)u^{1/p}\right) .
\end{equation*}%
Here, $f^{\ast }$ is the non-increasing rearrangement of the function $|f|$,
i.e., 
\begin{equation*}
f^{\ast }(u):=\inf \{s>0:\,\mu (\{t\in T:\,|f(t)|>s\})\leq u\},\text{ \ }%
0<u<\mu (T).
\end{equation*}%
Recall that $L^{p,p}=L^{p}$, $0<p<\infty $, isometrically.

Let $0<p_0<p_2<p_3<p_1<\infty $ and $0<q_0,q_1,q_2,q_3\leq \infty $ be
arbitrary. Since $L^{p_i,q_i}=(L^{r},L^{s})_{\theta_i,q_i}$ for $%
0<r<p_i<s<\infty$ and $1/p_i=(1-\theta_i )/r+\theta_i /s$ \cite[Theorem~5.3.1%
]{BL76}, then from Corollary \ref{Cor-admissibility} it follows that the
couple $(L^{p_0,q_0},L^{p_1,q_1})$ is mutually closed. Moreover, $%
L^{p_2,q_2}=(L^{p_0,q_0},L^{p_1,q_1})_{\theta ,q_2}$ and $%
L^{p_3,q_3}=(L^{p_0,q_0},L^{p_1,q_1})_{\eta ,q_3}$, where $1/p_2=(1-\theta
)/p_0+\theta /p_1$ and $1/p_3=(1-\eta )/p_0+\eta /p_1$ \cite[Theorem~5.3.1]%
{BL76}. Since $0<\theta <\eta <1$, by Theorem \ref{Ar-Cw1}, we have 
\begin{equation*}
Int^{KM}\left( L^{p_2,q_2},L^{p_3,q_3}\right) =Int^{KM}\left(
L^{p_0,q_0},L^{p_3,q_3}\right)\cap
Int^{KM}\left(L^{p_2,q_2},L^{p_1,q_1}\right).
\end{equation*}

\vskip0.2cm

Suppose now that $(T,\Sigma ,\mu )$ is the half-line $(0,\infty )$ with the
Lebesgue measure. Then, arguing similarly as in the paper \cite{CSZ}, one
can prove that the couple $\left( L^{p,q},L^{r,s}\right) $ has the uniformly Calder\'{o}n-Mityagin property for all $0<p<r\leq \infty $ and $0<q,s\leq \infty $.
Therefore, from Corollary \ref{Ar-Cw2-cor} it follows that for all $%
0<p_{0}<p_{2}<p_{3}<p_{1}\leq \infty $ and $0<q_{0},q_{1},q_{2},q_{3}\leq
\infty $ 
\begin{equation*}
Int\left( L^{p_{2},q_{2}},L^{p_{3},q_{3}}\right) =Int\left(
L^{p_{0},q_{0}},L^{p_{3},q_{3}}\right) \cap Int\left(
L^{p_{2},q_{2}},L^{p_{1},q_{1}}\right) .
\end{equation*}%
Observe that the last result was obtained in the case of $L^{p}$-spaces in \cite[Theorem~1.1]{CSZ}. 

\vskip0.2cm

\subsection{Weighted $L^{p}$-spaces.}

Let $0<p\leq \infty $ and $w(t)$ be a nonnegative measurable function on a $%
\sigma $-finite measure space $(T,\Sigma ,\mu )$. Consider a couple the
weighted $L^{p}$-spaces $\left( L^{p_{0}}\left( w_{0}\right)
,L^{p_{1}}\left( w_{1}\right) \right) $, $0<p_{0},p_{1}\leq \infty $ (in
particular, $p_{0}$ and $p_{1}$ may be equal). 

Let $0<\theta <\eta <1$. Assuming that $w_{2}=w_{0}^{1-\theta
}w\,_{1}^{\theta }$, $w_{3}=w_{0}^{1-\eta }w_{1}^{\eta }$, $%
1/p_{2}=(1-\theta )/p_{0}+\theta /p_{1}$ and $1/p_{3}=(1-\eta )/p_{0}+\eta
/p_{1}$, by \cite[Theorem~5.5.1]{BL76}, we get 
\begin{equation*}
\left( L^{p_{0}}\left( w_{0}\right) ,L^{p_{1}}\left( w_{1}\right) \right)
_{\theta ,p_{2}}=L^{p_{2}}\left( w_{2}\right) \;\;\mbox{and}\;\;\left(
L^{p_{0}}\left( w_{0}\right) ,L^{p_{1}}\left( w_{1}\right) \right) _{\eta
,p_{3}}=L^{p_{3}}\left( w_{3}\right) .
\end{equation*}%
In consequence, from Theorem \ref{Ar-Cw1} it follows 
\begin{equation*}
Int^{KM}\left( L^{p_{2}}\left( w_{2}\right) ,L^{p_{3}}\left( w_{3}\right)
\right) =Int^{KM}\left( L^{p_{0}}\left( w_{0}\right) ,L^{p_{3}}\left(
w_{3}\right) \right) \cap Int^{KM}\left( L^{p_{2}}\left( w_{2}\right)
,L^{p_{1}}\left( w_{1}\right) \right) .
\end{equation*}%
In particular, if $p_{0},p_{1}\geq 1$, by \cite{SP78} and Corollary \ref%
{Ar-Cw1-cor}, we obtain 
\begin{equation*}
Int\left( L^{p_{2}}\left( w_{2}\right) ,L^{p_{3}}\left( w_{3}\right) \right)
=Int\left( L^{p_{0}}\left( w_{0}\right) ,L^{p_{3}}\left( w_{3}\right)
\right) \cap Int\left( L^{p_{2}}\left( w_{2}\right) ,L^{p_{1}}\left(
w_{1}\right) \right) 
\end{equation*}%
(cf. \cite[Theorem~2.6]{BO06}). 

\vskip0.2cm

\subsection{Hardy spaces.}

Let $0<p\leq \infty $ and let $H^{p}:=H^{p}(\mathbb{R})$ be the (real) Hardy
space on the real line (see, for instance, \cite[§\,3.9.3]{BK91}). Then, as
is well known (see, for instance, \cite[Corollary~5.6.16, p.~374]{BSh}), for
all $0<\theta <1$ and $1\leq q\leq \infty $ we have 
\begin{equation*}
\left( H^{1},L^{\infty }\right) _{\theta ,q}=L^{p,q}
\end{equation*}%
where $1/p=1-\theta $. Moreover, the couple $\left( H^{1},L^{\infty }\right) 
$ has the uniform Calder\'{o}n-Mityagin property \cite{Sha88}. Consequently,
applying Corollary \ref{Ar-Cw2-cor} (see also Remark \ref{rem-Ar-Cw2}) for
all $1<p_{0}<p_{1}\leq \infty $ and $1\leq q_{0},q_{1}\leq \infty $ we get 
\begin{equation*}
Int\left( L^{p_{0},q_{0}},L^{p_{1},q_{1}}\right) =Int\left(
H^{1},L^{p_{1},q_{1}}\right) \cap Int\left( L^{p_{0},q_{0}},L^{\infty
}\right) .
\end{equation*}

Next, by \cite{FRS74} (see also \cite{Kis99}), for
all $0<p_{0}<p_{1}\le\infty $ and $0<\theta <1$ we have 
\begin{equation*}
\left( H^{p_{0}},H^{p_{1}}\right) _{\theta ,p}=H^{p},
\end{equation*}%
where $1/p=(1-\theta )/p_{0}+\theta /p_{1}$. Hence, from
Corollary \ref{Cor-admissibility} and Theorem \ref{Ar-Cw1} it follows that
for all $0<p_{0}<p_{2}<p_{3}<p_{1}\le\infty $ 
\begin{equation*}
Int^{KM}\left( H^{p_{2}},H^{p_{3}}\right) =Int^{KM}\left(
H^{p_{0}},H^{p_{3}}\right) \cap Int^{KM}\left( H^{p_{2}},H^{p_{1}}\right). 
\end{equation*}%
In particular, since the couple $\left( H^{1},H^{\infty
}\right) $ has the uniform Calder\'{o}n-Mityagin property \cite{Jon84} (see also \cite[p.~684]{BK91}), from Corollary \ref{Ar-Cw2-cor} and Remark \ref{rem-Ar-Cw2} it follows 
\begin{equation*}
Int\left( H^{p},H^{q}\right) =Int\left( H^{1},H^{q}\right) \cap Int\left(
H^{p},H^{\infty }\right) 
\end{equation*}%
for all $1<p<q<\infty .$ 

\vskip0.2cm

\subsection{Besov spaces.}

Let $B_{p,q}^{s}:=B_{p,q}^{s}(\mathbb{R}^{n})$ be the Besov spaces on $%
\mathbb{R}^{n}$ (see, for instance, \cite[§\,2.3]{Tr83}). Suppose $-\infty
<s_{0}<s_{1}<\infty $, $0<\theta <1$, $s=(1-\theta )s_{0}+\theta s_{1}$, $%
0<q_{0},q_{1},q\leq \infty $. Then, for all $0<p\leq \infty $ 
\begin{equation*}
\left( B_{p,q_{0}}^{s_{0}},B_{p,q_{1}}^{s_{1}}\right) _{\theta
,q}=B_{p,q}^{s}
\end{equation*}%
(see \cite[Theorem~2.4.2]{Tr83}). Hence, from Corollary \ref%
{Cor-admissibility} and Theorem \ref{Ar-Cw1} it follows that for all $\infty
<s_{0}<s_{2}<s_{3}<s_{1}<\infty $ and $0<q_{0},q_{1},q_{2},q_{3}\leq \infty $
we have 
$$
Int^{KM}\left( B_{p,q_{2}}^{s_{2}},B_{p,q_{3}}^{s_{3}}\right)
=Int^{KM}\left( B_{p,q_{0}}^{s_{0}},B_{p,q_{3}}^{s_{3}}\right) \cap
Int^{KM}\left( B_{p,q_{2}}^{s_{2}},B_{p,q_{1}}^{s_{1}}\right) .$$

\vskip0.2cm

\subsection{Lorentz quasi-normed ideals of compact operators.}

Let ${H}$ be a separable complex Hilbert space. For every compact operator $A:\,H\to H$ let ${s}(A)=\{s_{n}({A})\}_{n=1}^{\infty }$ be the sequence of $%
s$-numbers of ${A}$ determined by the Schmidt expansion \cite{GK}. For every 
$p,q>0$, the class ${{\mathfrak{S}}}^{p,q}$ consists of all compact
operators ${A}:\,{H}\rightarrow {H}$ such that 
\begin{equation*}
{\Vert A\Vert }_{p,q}:={\Vert {s}(A)\Vert }_{l^{p,q}}<\infty ,
\end{equation*}%
where $l^{p,q}$, the discrete version of the function space $L^{p,q},$
consists of all sequences ${x}=(x_{n})_{n=1}^{\infty }$ of real numbers such
that the quasi-norm 
\begin{equation*}
{\Vert {x}\Vert }_{l^{p,q}}:=\left( \sum_{n=1}^{\infty }(x_{n}^{\ast
})^{q}n^{\frac{q}{p}-1}\right) ^{\frac{1}{q}}
\end{equation*}%
is finite. Here, ${x}^{\ast }=(x_{n}^{\ast })_{n=1}^{\infty }$ denotes the
nonincreasing permutation of the sequence $(|x_{n}|)_{n=1}^{\infty }$.

The classes ${{\mathfrak{S}}}^{p,q},$ $p,q>0,$ are two-sided symmetrically
quasi-normed ideals in the space of all bounded operators in ${H,}$ which
sometimes are referred to as Lorentz ideals. The classical Schatten-von
Neumann ideals ${{\mathfrak{S}}}^{p}$ correspond to the case $p=q,$ i.e., ${{%
\mathfrak{S}}}^{p}={{\mathfrak{S}}}^{p,p}$.

Let $0<p_{0}<p_{2}<p_{3}<p_{1}<\infty $ and $0<q_{0},q_{1},q_{2},q_{3}\leq
\infty $ be arbitrary. Since ${{\mathfrak{S}}}^{p_{i},q_{i}}=({{\mathfrak{S}}%
}^{r},{{\mathfrak{S}}}^{s})_{\theta _{i},q_{i}}$ for $0<r<p_{i}<s<\infty $
and $1/p_{i}=(1-\theta _{i})/r+\theta _{i}/s$ (see, for instance, \cite[§%
\,7.3]{BL76}, then from Corollary \ref{Cor-admissibility} it follows that
the couple $({{\mathfrak{S}}}^{p_{0},q_{0}},{{\mathfrak{S}}}^{p_{1},q_{1}})$
is mutually closed. Moreover, ${{\mathfrak{S}}}^{p_{2},q_{2}}=({{\mathfrak{S}%
}}^{p_{0},q_{0}},{{\mathfrak{S}}}^{p_{1},q_{1}})_{\theta ,q_{2}}$ and ${{%
\mathfrak{S}}}^{p_{3},q_{3}}=({{\mathfrak{S}}}^{p_{0},q_{0}},{{\mathfrak{S}}}%
^{p_{1},q_{1}})_{\eta ,q_{3}}$, where $1/p_{2}=(1-\theta )/p_{0}+\theta
/p_{1}$ and $1/p_{3}=(1-\eta )/p_{0}+\eta /p_{1}$ \cite[Theorem~5.3.1]{BL76}%
. Since $0<\theta <\eta <1$, by Theorem \ref{Ar-Cw1}, we have 
\begin{equation*}
Int^{KM}\left( {{\mathfrak{S}}}^{p_{2},q_{2}},{{\mathfrak{S}}}%
^{p_{3},q_{3}}\right) =Int^{KM}\left( {{\mathfrak{S}}}^{p_{0},q_{0}},{{%
\mathfrak{S}}}^{p_{3},q_{3}}\right) \cap Int^{KM}\left( {{\mathfrak{S}}}%
^{p_{2},q_{2}},{{\mathfrak{S}}}^{p_{1},q_{1}}\right) .
\end{equation*}

\end{document}